\documentclass[final,1p,times]{elsarticle}
\usepackage{mathrsfs}
\usepackage{amscd,bbm,amsmath,amsthm,graphicx,amssymb,mathrsfs,latexsym,dsfont,amsfonts,paralist,epsfig,graphics,exscale,ulem}
\usepackage[colorlinks=true]{hyperref}
\def\bx{\boldsymbol{x}}
\def\by{\boldsymbol{y}}
\newtheorem{Mthm}{Theorem}%

\newtheorem{theorem}{Theorem}[section]

\newtheorem{lemma}[theorem]{Lemma}

\theoremstyle{definition}

\newtheorem{defn}[theorem]{Definition}
\newtheorem*{definition}{Definition}

\journal{Mathematische Zeitschrift $($accepted$)$}
\begin{document}

\begin{frontmatter}

\title{On $\mathrm{C}^1$-class local diffeomorphisms whose periodic points are nonuniformly expanding\tnoteref{label1}}
\tnotetext[label1]{Project was supported partly by National Natural Science Foundation of China (No.~11071112) and PAPD of Jiangsu Higher Education Institutions.}%
 \author{Xiongping Dai}
 \address{Department of Mathematics, Nanjing University, Nanjing 210093, People's Republic of China}
 \ead{xpdai@nju.edu.cn}
\begin{abstract}
Using a sifting-shadowing combination, we prove in this paper that an arbitrary $\mathrm{C}^1$-class
local diffeomorphism $f$ of a closed manifold $M^n$ is uniformly expanding on the
closure $\mathrm{Cl}_{M^n}(\mathrm{Per}(f))$ of its periodic point set $\mathrm{Per}(f)$, if it is
nonuniformly expanding on $\mathrm{Per}(f)$.
\end{abstract}

\end{frontmatter}
\section{Introduction}\label{sec1}%

We consider a discrete-time differentiable semi-dynamical system
\begin{equation*}
f\colon M^n\rightarrow M^n
\end{equation*}
which is a $\mathrm{C}^1$-class local diffeomorphism of a closed manifold
$M^n$, where $n\ge1$; that is to say, $f$ is surjective and for any $x\in M^n$, there is an open neighborhood $U_x$ around $x$ in $M^n$ such that $f$ is $\mathrm{C}^1$-class diffeomorphic restricted to $U_x$.

\subsection{Motivation}\label{sec1.1}%

A point $p\in M^n$ is said to be \textit{periodic with period} $\tau\ge1$, if $f^{\tau}(p)=p$. For a number of situations in smooth ergodic theory and
differentiable dynamical systems, the ``nonuniform hyperbolicity" of
its periodic point set, written as $\mathrm{Per}(f)$, is often proven or assumed;
for example, see the classical works~\cite{Fr71, L80, Man82, Man87, Aok92, Hay92, GW}. Then,
extending the hyperbolicity from the periodic points to the whole manifold or the closure of $\mathrm{Per}(f)$ is a
deep and important problem.

In this paper, we are concerned with the study of conditions for a
nonuniformly expanding endomorphism to be uniformly expanding. There have been a
few results concerning this. One of these results is the remarkable
Theorem~A of Ma\~{n}\'{e}~\cite{Man85} for
$\mathrm{C}^{1+{\tiny\textrm{H\"{o}lder}}}$ endomorphisms of the unit circle $\mathbb{T}^1$; some
interesting other results for $\mathrm{C}^1$-class local diffeomorphisms of $M^n$ where $n\ge2$, have appeared in several
recent papers \cite{AAS,Cao,CLR,COP,Dai09-2}.

\subsection{Basic concepts}\label{sec1.2}%
Before we pursue a further discussion, let us first recall some basic concepts. As usual, by $\textmd{Diff}_{\textrm{loc}}^1(M^n)$ we denote the set of all
$\mathrm{C}^1$-class local diffeomorphisms of the closed manifold $M^n$, equipped with the usual $\mathrm{C}^1$-topology.

By an abuse of notation, $\|\cdot\|_{\mathrm{co}}$ means the co-norm
(also called minimum norm) defined in the following way: for any $f\in\textmd{Diff}_{\textrm{loc}}^1(M^n)$,
\begin{equation*}
\|D_{\!x}f^\ell\|_{\mathrm{co}}=\min_{v\in T_{\!x}{M^n}, \|v\|=1}\|(D_{\!x}f^\ell)v\|
\end{equation*}
for the derivatives $D_{\!x}f^\ell\colon T_{\!x}M^n\rightarrow T_{\!f^\ell(x)}M^n$ for all $x\in
M^n$ and $\ell\ge1$. Since $f$ is locally diffeomorphic, $\|D_{\!x}f^\ell\|_{\mathrm{co}}=\|(D_{\!x}f^\ell)^{-1}\|^{-1}>0$. On the other hand, we have
\begin{equation*}
\|D_{\!x}f^{\ell+m}\|_{\mathrm{co}}\ge\|(D_{\!x}f^\ell)\|_{\mathrm{co}}\cdot\|(D_{\!f^\ell(x)}f^m\|_{\mathrm{co}}
\end{equation*}
for any $x\in M^n$ and any $\ell,m\ge1$.

For any point $x\in M^n$, let
\begin{equation*}
\lambda_{\min}(x,f)=\limsup_{\ell\to+\infty}\frac{1}{\ell}\log\|D_{\!x}f^\ell\|_{\mathrm{co}}
\end{equation*}
be the \textit{minimal Lyapunov exponent} of $f$ at the base point $x$. Our question considered here is this: If $\lambda_{\min}(x,f)>0$ for all $x\in\mathrm{Per}(f)$, whether $f$ is expanding on the closure of $\mathrm{Per}(f)$. Let us first see an example. We denote by $\{0,1\}^\mathbb{N}$ the compact topological space of all the one-sided infinite sequences $i_{\boldsymbol{\cdot}}\colon\mathbb{N}\rightarrow\{0,1\}$ and let $\boldsymbol{S}(\alpha,\gamma)=\{S_0,S_1\}$ where
\begin{equation*}
S_0=\alpha\left[\begin{array}{cc}1&-1\\0&1\end{array}\right],\quad S_1=\gamma\left[\begin{array}{cc}1&0\\-1&1\end{array}\right]\qquad (\alpha,\gamma>1).
\end{equation*}
Based on \cite{BM02}, there exists a pair of $\alpha,\gamma$ such that for every periodic sequence $i_{\boldsymbol{\cdot}}\in\{0,1\}^\mathbb{N}$
\begin{equation*}
\lambda_{\min}(i_{\boldsymbol{\cdot}},\boldsymbol{S}(\alpha,\gamma))=\lim_{n\to+\infty}\frac{1}{n}\log\|S_{i_1}\cdots S_{i_n}\|_{\mathrm{co}}>0,
\end{equation*}
but the linear cocycle, associated to $\boldsymbol{S}(\alpha,\gamma)$ and driven by the Markov shift $\theta\colon i_{\boldsymbol{\cdot}}\mapsto i_{\boldsymbol{\cdot}+1}$, is not expanding on $\{0,1\}^\mathbb{N}$, although all the periodic sequences $i_{\boldsymbol{\cdot}}$ form a dense subset of $\{0,1\}^\mathbb{N}$.

This example motivates us to have to strengthen condition for uniformly expanding. The basic condition that
we study in this paper is described as follows.

\begin{definition}
We say that $f\in \mathrm{Diff}_{\mathrm{loc}}^1(M^n)$ is
{\it nonuniformly expanding} on a subset $\Lambda\subseteq M^n$ not
necessarily closed, if there can be found a number $\lambda>0$ such
that
\begin{equation*}
\limsup_{k\to+\infty}\frac{1}{k}\sum_{i=0}^{k-1}\log\|D_{\!f^i(x)}f\|_{\mathrm{co}}\ge\lambda
\end{equation*}
for all $x\in\Lambda$. Here the constant $\lambda$ is called an \textit{expansion indicator} of $f$ at $\Lambda$.
\end{definition}

This is similar to what Alves, Bonatti and Viana has defined in \cite{ABV}. Recall that for an arbitrary ergodic measure $\mu$ of $f$, based on the Kingman subadditive ergodic theorem~\cite{King} one can introduce the {\it minimal Lyapunov exponent of $f$
restricted to $\mu$} by
\begin{equation*}
\lambda_{\min}(\mu,f)=\lim_{\ell\to+\infty}\frac{1}{\ell}\log\|D_{\!x}f^\ell\|_{\mathrm{co}}\qquad\mu\textrm{-a.e. }x\in M^n.
\end{equation*}
It is worthwhile noting here that restricted to a subset $\Lambda$, the nonuniformly expanding
property of $f$ is more stronger than the condition that $f$
has only Lyapunov exponents $\lambda_{\min}(\mu,f)$, which are positive and uniformly bounded away from zero, for all ergodic measures $\mu$ in $\Lambda$. So, if $f$ is nonuniformly expanding on $\Lambda$, then every ergodic measure $\mu$ of $f$
distributed in $\Lambda$ has only positive Lyapunov exponent $\lambda_{\min}(\mu,f)$. But,
because here $\Lambda$ is not necessarily a closed subset of $M^n$
such as $\Lambda=\mathrm{Per}(f)$, it is possible that there exists
some ergodic measure $\mu$ which is just supported on the boundary
$\partial\Lambda$, not in $\Lambda$, and which cannot be a-priori approximated
arbitrarily by periodic measures even in the case $\Lambda=\mathrm{Per}(f)$. This prevents us from using the
expanding criteria already developed, for example, in~\cite{AAS,Cao} and \cite[Lemma I-5]{Man87}, to prove that
$f$ is uniformly expanding on the closure $\mathrm{Cl}_{M^n}(\Lambda)$ of $\Lambda$ in $M^n$.

\subsection{Main results and outlines}\label{sec1.3}%

Our principal result obtained in this paper can be stated as
follows, which will be proved in Section~\ref{sec5.1} based on a series of lemmas developed in Sections~\ref{sec2}, \ref{sec3} and \ref{sec4}.

\begin{Mthm}\label{thm1}
Let $f\colon M^n\rightarrow M^n$ be a $\mathrm{C}^1$-class local diffeomorphism on the closed manifold $M^n$,
which is nonuniformly expanding on its periodic point set
$\mathrm{Per}(f)$. Then, there hold the following statements.
\begin{description}
\item[(1)] $f$ is uniformly expanding on the closure
$\mathrm{Cl}_{M^n}(\mathrm{Per}(f))$, i.e., there can be found two numbers
$C>0$ and $\lambda>0$ such that
\begin{equation*}
\|(D_{\!x}f^k)v\|\ge C\|v\|\exp(k\lambda)
\end{equation*}
for all $v\in T_{\!x}{M^n}, x\in\mathrm{Cl}_{M^n}(\mathrm{Per}(f))$ and $k\ge 1$.

\item[(2)] For an arbitrary ergodic measure $\mu$ of $f$, either $\lambda_{\min}(\mu,f)\le0$, or $\lambda_{\min}(\mu,f)\ge\lambda$ and
$\mathrm{supp}(\mu)\subseteq\mathrm{Cl}_{M^n}(\mathrm{Per}(f))$.

\item[(3)] If additionally $\mathrm{Per}(f)$ is dense in the
nonwandering point set $\Omega(f)$ of $f$, then $f$ is uniformly expanding on $M^n$.
\end{description}
\end{Mthm}

The statement (1) of Theorem~\ref{thm1} is closely related to an important theorem of Ma\~{n}\'{e}~\cite[Theorem~$\mathrm{II}$-1]{Man87}, which essentially reads as follows: If $f\in\mathrm{Diff}^1(M^n)$ preserves a homogeneous dominated splitting $T_{\!\Delta}M^n=E\oplus F$ where $\Delta=\mathrm{Cl}_{M^n}(\mathrm{Per}(f))$, such that the bundle $E$ is contracted by $Df$ and at every periodic point $p$,
\begin{equation*}
\limsup_{k\to+\infty}\frac{1}{k}\sum_{i=0}^{k-1}\log\|(D_{\!{f^i(p)}}f)|F\|_{\mathrm{co}}\ge\lambda
\end{equation*}
for some uniform constant $\lambda>0$, then $f$ is (uniformly) expanding along $F$ on $\Delta$. However, Ma\~{n}\'{e}'s theorem does not apply directly to our situation studied here, since $f$ is not a diffeomorphism. In addition, the statement (1) of Theorem~\ref{thm1} is proved by Castro, Oliveira and Pinheiro~\cite{COP} in the special case where $f$ possesses the closing by periodic orbits property, and by Sun and Tian~\cite{ST} in the generic case.

For any $f\in\mathrm{Diff}_{\mathrm{loc}}^1(M^n)$, by definition, the nonuniformly expanding for $f$ on $\mathrm{Per}(f)$ is equivalent to the property that
there exists a constant $\lambda>0$ such that
\begin{equation*}
\int_{M^n}\log\|D_{\!x}f\|_{\mathrm{co}}\,\mu(dx)\ge\lambda
\end{equation*}
{\it for all ergodic measures $\mu$ of $f$ supported on periodic orbits}. However, from R.~Ma\~{n}\'{e}~\cite[Lemma I-5]{Man87} it follows that, $f$ is (uniformly) expanding on $\mathrm{Cl}_{M^n}(\mathrm{Per}(f))$
if and only if there exists an integer $m\ge1$ and a constant $\lambda^\prime>0$ such that
\begin{equation*}
\int_{M^n}\log\|D_{\!x}f^m\|_{\mathrm{co}} \mu(dx)\ge\lambda^\prime
\end{equation*}
{\it for all ergodic measures $\mu$ of $f$ supported on $\mathrm{Cl}_{M^n}(\mathrm{Per}(f))$}.
Since $\mathrm{Per}(f)$ does not need to be closed in $M^n$ and there is no
a-priori generic condition, like closing by periodic orbits property, for the restriction of $f$ to
$\mathrm{Cl}_{M^n}(\mathrm{Per}(f))$ to ensure that each ergodic measure of $f$ distributed on $\mathrm{Cl}_{M^n}(\mathrm{Per}(f))$ can be arbitrarily approximated by periodic ones, Ma\~{n}\'{e}'s criterion above does not work here.
We will prove the uniformly expanding property by employing
a Liaowise ``sifting-shadowing combination" motivated by
S.-T.~Liao~\cite{L80} and R.~Ma\~{n}\'{e}~\cite{Man87}.

To overcome the
non-invertibility of $f$, we will introduce the natural extension of
$f$ in Section~\ref{sec3}. The idea of the proof of Theorem~\ref{thm1} is that if $f$ had not been uniformly
expanding on $\mathrm{Cl}_{M^n}(\mathrm{Per}(f))$ then, using the natural extension of $f$ and a sifting lemma (Pliss lemma),
we would construct an ``abnormal" quasi-expanding pseudo-orbit string
of $f$ in $\mathrm{Cl}_{M^n}(\mathrm{Per}(f))$. Further, a shadowing lemma (Theorem~\ref{thm2.1}) enables us to find an ``abnormal" periodic orbit $P$ whose minimal
Lyapunov exponent $\lambda_{\min}(P,f)$ can approach arbitrarily to zero (Theorem~\ref{thm4.1}), which
contradicts the nonuniformly expanding property of $f$ on $\mathrm{Per}(f)$. As our sifting-shadowing
combination, here we use the Pliss lemma (Lemma~\ref{lem2.2}) and the shadowing lemma (Theorem~\ref{thm2.1}) proved in Appendix in Section~\ref{sec6}.

In the context of the stability conjecture of Palis and Smale, Pliss~\cite{Pli}, Liao~\cite{L80} and Ma\~{n}\'{e}~\cite{Man82} were independently led to the notion of dominated splitting of the tangent bundle into two subbundles: one of them is definitely more contracted (or less expanded) than the other, after a uniform number of iterates.
Recall from \cite{L80,BDV} that for any $f\in\mathrm{Diff}_\mathrm{loc}^1(M^n)$ and $\eta>0$, we say $f$ has an ``$(\eta,1)$-dominated splitting" over $\mathrm{Cl}_{M^n}(\mathrm{Per}(f))$, provided that there exists a constant $C>0$ and $Df$-invariant decomposition of $T_{\!\mathrm{Cl}_{M^n}(\mathrm{Per}(f))}M^n$ into two subbundles
\begin{equation*}
T_{\!x}M^n=E(x)\oplus F(x)\quad \textrm{with }\dim E(x)=1\qquad\forall x\in\mathrm{Cl}_{M^n}(\mathrm{Per}(f))
\end{equation*}
such that
\begin{equation*}
\frac{\|D_{\!x}f^k|{E(x)}\|}{\;\;\,\|D_{\!x}f^k|{F(x)}\|_{\mathrm{co}}}\le C\exp(-2\eta k)\qquad\forall k\ge1.
\end{equation*}
By choosing an adapted norm, there is no loss of generality in assuming $C=1$ for simplicity.

As a result of Theorem~\ref{thm1}, we will obtain the following statement in Section~\ref{sec5.2}.

\begin{Mthm}\label{thm2}
Let $f\colon M^n\rightarrow M^n$ be a $\mathrm{C}^1$-class local diffeomorphism where $n\ge2$, and assume $f$ possesses an $(\eta,1)$-dominated splitting over $\mathrm{Cl}_{M^n}(\mathrm{Per}(f))$.
If every $p\in\mathrm{Per}(f)$ have only positive Lyapunov exponents and such exponents are uniformly bounded away from $0$,
then $f$ is uniformly expanding on $\mathrm{Cl}_{M^n}(\mathrm{Per}(f))$.
\end{Mthm}

Finally, in Section~\ref{sec5.3}, we will apply Theorem~\ref{thm1} stated above
to a $\mathrm{C}^1$-class local diffeomorphism of the circle $\mathbb{T}^1$; see Theorem~\ref{thm3} below.
\section{Closing property of recurrent quasi-expanding orbit
strings}\label{sec2}

To apply a sifting-shadowing combination, we need first to introduce
a suitable shadowing lemma and a sifting lemma for local
diffeomorphisms of the closed manifold $M^n$. For that, we have to introduce two notions:
``$\lambda$-quasi-expanding orbit-string" and ``shadowing property" of
quasi-expanding pseudo-orbit.
\subsection{Closing up quasi-expanding strings}\label{sec2.1}
Consider an arbitrary $f\in
\textmd{Diff}_{\mathrm{loc}}^1(M^n)$. Recall that for any
$\lambda>0$, an ordered segment of orbit of $f$ of length $k$
\begin{equation*}
\left(x,f^k(x)\right):=\left(x,f(x),\ldots,f^k(x)\right)\qquad
(k\ge 1)
\end{equation*}
is called a {\it $\lambda$-quasi-expanding orbit-string of $f$} if
\begin{equation*}
\frac{1}{\ell}\sum_{j=1}^{\ell}\log\|D_{\!f^{k-j}(x)}f\|_{\mathrm{co}}\geq\lambda
\qquad\forall \ell=1, \ldots, k.
\end{equation*}

As a complement to Liao's shadowing lemma~\cite{L79}, we could obtain the
following shadowing lemma.

\begin{theorem}\label{thm2.1}
Given any $f\in\mathrm{Diff}_{\mathrm{loc}}^1(M^n)$ and any two numbers
$\varepsilon>0$ and $\lambda>0$, there exists a number
$\delta=\delta(\varepsilon,\lambda)>0$ such that, if all
$\left(x_i,f^{n_i}(x_i)\right),\ i=0,\ldots,k$, are
$\lambda$-quasi-expanding orbit-strings satisfying
$\textsl{d}\left(f^{n_i}(x_i),x_{i+1}\right)<\delta$ for all $0\le
i\le k$ where $x_{k+1}=x_0$, then there can be found a periodic
point $\bx$ of $f$ with period $\tau_{\bx}=n_0+\cdots+n_k$ verifying
\begin{equation*}
\textsl{d}(f^{n_{-1}+\cdots+n_{i-1}+j}(\bx),f^j(x_{i}))<\varepsilon\qquad
0\le j\le n_{i},\ 0\le i<k
\end{equation*}
and
\begin{equation*}
\frac{1}{\ell}\sum_{j=1}^{\ell}\log\|D_{\!f^{\tau_{\bx}-j}(\bx)}f\|_{\mathrm{co}}\geq\lambda-\varepsilon\qquad\forall \ell=1,
\ldots, \tau_{\bx}.
\end{equation*}
\end{theorem}

\noindent Here $n_{-1}=0$ and $\textsl{d}(\cdot,\cdot)$ is an arbitrarily
preassigned natural metric on $M^n$.

In fact, following the ideas of \cite{Kat80, Dai-TMJ}, one can further obtain an $(\varepsilon,\rho)$-exponential closing property under this situation. Here the proof of this theorem is standard following \cite{G}; see Appendix below for the
details.
\subsection{The Pliss lemma}\label{sec2.2}%
For our sifting lemma, we shall apply the following
reformulation of a result due to V.~Pliss.

\begin{lemma}[\cite{Pli}]\label{lem2.2}
Let $H>0$ be arbitrarily given. For any $\gamma>\gamma^\prime>0$, there exists an
integer $N_{\gamma,\gamma^\prime}\ge 1$ and a real number
$c_{\gamma,\gamma^\prime}\in(0,1)$ such that, if $(a_0,\ldots,
a_{m-1})$ with $m\ge N_{\gamma,\gamma^\prime}$ and $|a_i|\le H$ for all $0\le i<m$, is
a ``$\gamma$-string" in the sense that
\begin{equation*}
\frac{1}{m}\sum_{i=0}^{m-1}a_i\geq\gamma,
\end{equation*}
then there can be found integers $0<n_1<\cdots<n_k\le m$ with $k\ge\max\{1, mc_{\gamma,\gamma^\prime}\}$ such that
$(a_0,\ldots,a_{n_i-1})$ is a ``$\gamma^\prime$-quasi-expanding
string" for all $1\le i\le k$, i.e.,
\begin{equation*}
\frac{1}{J}\sum_{j=1}^{J}a_{n_i-j}\geq\gamma^\prime\qquad
\forall J=1,\ldots,n_i.
\end{equation*}
\end{lemma}

We note here that in the above Pliss lemma, the numbers $N_{\gamma,\gamma^\prime}$ and
$c_{\gamma,\gamma^\prime}$ both depend on the preassigned constant $H$. For our applications later, we will consider the special case
where
\begin{equation*}
H=\max\left\{|\log\|D_{\!x}f\|_{\mathrm{co}}|;\ x\in
M^n\right\}
\quad \textrm{and}\quad
a_i=\log\|D_{\!f^i(x)}f\|_{\mathrm{co}}
\end{equation*}
for a local diffeomorphism $f$ and $x\in M^n$.

By a so-called \textit{sifting-shadowing combination}, we mean a combinatorial
application of a sifting lemma like Lemma~\ref{lem2.2} and a
shadowing lemma like Theorem~\ref{thm2.1}. It is an effective strategy to prove
hyperbolicity in differentiable dynamical systems, see~\cite{L80, GW, Dai11} for example.

\subsection{Existence of periodic repellers}\label{sec2.3}%
Using Theorem~\ref{thm2.1} and Lemma~\ref{lem2.2} for a $\mathrm{C}^1$-class local diffeomorphism $f$, we can obtain the following theorem on existence of periodic repellers under the assumption that $f$ preserves an expanding ergodic measure. This theorem will be needed in the proof of the statement (2) of Theorem~\ref{thm1}.

\begin{theorem}\label{thm2.3}
Let $f\in\mathrm{Diff}_{\mathrm{loc}}^1(M^n)$ preserve an ergodic probability measure $\mu$. If the minimal Lyapunov exponent of $f$ restricted to $\mu$
\begin{equation*}
\lambda_{\min}(\mu,f)=\lim_{\ell\to+\infty}\frac{1}{\ell}\log\|D_{\!x}f^\ell\|_{\mathrm{co}}>0\qquad\mu\textrm{-a.e. }x\in M^n,
\end{equation*}
then for any $0<\gamma^{\prime\prime}<\lambda_{\min}(\mu,f)$, there exists a sequence of periodic repellers $\{P_\ell\}_1^\infty$ of $f$ with
\begin{equation*}
\lim_{\ell\to+\infty}P_\ell=\mathrm{supp}(\mu)
\end{equation*}
in the sense of the Hausdorff topology, such that $\bigcup_\ell P_\ell$ is nonuniformly expanding for $f^\kappa$ with an expansion indicator $\gamma^{\prime\prime}k$, for some $\kappa\ge1$.
\end{theorem}

We notice here that, if $f\in\mathrm{Diff}^1(M^n)$ preserves an ergodic probability measure $\mu$ satisfying
\begin{equation*}
\lambda_{\max}(\mu,f)=\lim_{\ell\to+\infty}\frac{1}{\ell}\log\|D_{\!x}f^\ell\|<0\qquad\mu\textrm{-a.e. }x\in M^n,
\end{equation*}
then Liao proved, using his theory of standard systems of equations in \cite{L73}, that $\mu$ is supported on a periodic attractor of $f$. Our Theorem~\ref{thm2.3} is thus an extension of Liao's result.

To prove this theorem, we need the following subadditive version of
\cite[Theorem~2]{DZ}, which guarantees the existence of a long $\gamma$-string. This long $\gamma$-string enables us to use the Pliss lemma (Lemma~\ref{lem2.2}) and then the shadowing lemma (Theorem~\ref{thm2.1}).

\begin{lemma}\label{lem2.4}
Let $\theta\colon X\rightarrow X$ be a discrete-time
semidynamical system of a compact metrizable space
$X$, which preserves a Borel probability measure $\mu$, and $\{t_\ell\}_{\ell=1}^{+\infty}$ an integer sequence with
\begin{equation*}
t_1\ge 1,\ t_{\ell+1}=2t_\ell \quad (\ell=1,2,\dotsc).
\end{equation*}
Let $\varphi\colon \mathbb{N}\times X\rightarrow \mathbb{R}\cup\{-\infty\}$ be a
measurable function with the subadditivity property
\begin{equation*}
\varphi(k_1+k_2, x)\le\varphi(k_1,x)+\varphi(k_2,\theta^{k_1}(x))\quad \mu\textrm{-a.e. } x\in X
\end{equation*}
for any $k_1,k_2\ge 1$, such that
\begin{description}
\item[(a)] $\varphi(t,\cdot)\in{\mathscr{L}}^1(X,\mu)$ for all $t\ge1$, and

\item[(b)] $\{t^{-1}\varphi(t,\cdot)\}_{t=1}^{+\infty}$ is bounded below by an $\mu$-integrable function.
\end{description}
Then, there exists a Borel subset of $\mu$-measure $1$, written as
$\widehat{\varGamma}$, such that for all $x\in\widehat{\varGamma}$,
\begin{equation*}
\bar{\varphi}(x)=\lim_{t\to+\infty}\frac{1}{t}\varphi(t,x)\qquad
\textrm{with }\bar{\varphi}(\theta^t(x))=\bar{\varphi}(x)\ \forall t>0
\end{equation*}
and
\begin{equation*}
\lim_{\ell\to+\infty}\left\{\lim_{k\to+\infty}\frac{1}{k}\sum_{j=0}^{k-1}\frac{1}{t_\ell}\varphi(t_\ell,\theta^{jt_\ell}(x))\right\}=\bar{\varphi}(x).
\end{equation*}
\end{lemma}

\noindent \textbf{Note.} Here $\bar{\varphi}(x)$ is defined by the Kingman
subadditive ergodic theorem~\cite{King} such that
\begin{equation*}
\int_X\bar{\varphi}(x)\mu(dx)=\lim_{t\to+\infty}\frac{1}{t}\int_X\bar{\varphi}(t,x)\mu(dx)=\inf_{t\ge1}\frac{1}{t}\int_X\bar{\varphi}(t,x)\mu(dx).
\end{equation*}
Since
\begin{equation*}
|\bar{\varphi}(x)|=\lim_{t\to+\infty}\frac{1}{t}|\varphi(t,x)|\le\lim_{t\to+\infty}\frac{1}{t}\sum_{i=0}^{t-1}|\varphi(1,\theta^i(x))|=\psi(x)\qquad\mu\textrm{-a.e. }x\in X,
\end{equation*}
where $\psi\in \mathscr{L}^1(X,\mu)$ is defined by the Birkhoff ergodic theorem for $\theta\colon X\rightarrow X$ and $|\varphi(1,\cdot)|$, hence under our hypothesis we have
$\bar{\varphi}\in\mathscr{L}^1(X,\mu)$.

\begin{proof}
The following argument is parallel to that of \cite[Theorem~2]{DZ}.
According to the Kingman subadditive ergodic theorem, there is a
Borel set $\varGamma^\prime\subset X$ of $\mu$-measure $1$ and a measurable function $\bar{\varphi}\in\mathscr{L}^1(X,\mu)$ such that
\begin{equation*}
\bar{\varphi}(x)=\lim_{t\to+\infty}\frac{1}{t}\varphi(t,x)\quad
\textrm{with }\bar{\varphi}(\theta^t(x))=\bar{\varphi}(x)\qquad
\forall x\in\varGamma^\prime
\end{equation*}
and
\begin{equation*}
\int_X\bar{\varphi}(x)\mu(dx)=\lim_{t\to+\infty}\frac{1}{t}\int_X\varphi(t,x)\mu(dx).
\end{equation*}
So, for the given sequence $\{t_\ell\}_1^\infty$ we have
\begin{equation*}
\lim_{\ell\to+\infty}\int_X\frac{1}{t_\ell}\varphi(t_\ell,\theta^{\alpha}(x))\mu(dx)=\int_X\bar{\varphi}(x)\mu(dx)
\end{equation*}
uniformly for $\alpha\in\mathbb{Z}_+$, since $\mu$ is $\theta^\alpha$-invariant. For any
$\varepsilon>0$, there thus exists an $\ell(\varepsilon)>0$ such that, if
$\ell\ge\ell(\varepsilon)$ then
\begin{equation*}
\frac{1}{k}\sum_{j=0}^{k-1}\int\limits_X\left\{\frac{1}{t_\ell}\varphi(t_\ell,\theta^{jt_\ell}(x))-\bar{\varphi}(x)\right\}\mu(dx)\le\frac{\varepsilon}{2}\qquad\forall k\in\mathbb{N}.
\end{equation*}
This means that for all $\ell\ge\ell(\varepsilon)$ there holds the
inequality
\begin{equation*}
\int\limits_X\frac{1}{k}\sum_{j=0}^{k-1}\left\{\frac{1}{t_\ell}\varphi(t_\ell,\theta^{jt_\ell}(x))-\bar{\varphi}(x)\right\}\mu(dx)\le\frac{\varepsilon}{2}\qquad\forall k\in\mathbb{N}.
\end{equation*}

For any $\ell\ge1$, we now consider the sample
$\theta^{t_\ell}\colon X\rightarrow X$ which also preserves $\mu$, but not necessarily ergodic even if $\mu$ is ergodic under $\theta\colon X\rightarrow X$.
From the Birkhoff ergodic theorem and the subadditivity of
$\varphi(t,x)$, it follows that there is a
$\theta^{t_\ell}$-invariant Borel subset
$W_\ell^*\subset\varGamma^\prime$ of $\mu$-measure $1$ such that
\begin{equation*}
\lim_{k\to+\infty}\frac{1}{k}\sum_{j=0}^{k-1}\left\{\frac{1}{t_\ell}\varphi(t_\ell,\theta^{jt_\ell}(x))-\bar{\varphi}(x)\right\}=h_\ell^*(x)\ge
0\qquad\forall x\in W_\ell^*
\end{equation*}
for some $h_\ell^*(\cdot)\in{\mathscr{L}}^1(X,\mu)$, for all $\ell\ge1$. Set
\begin{equation*}
\widehat{\varGamma}=\bigcap_{\ell=1}^{+\infty} W_\ell^*.
\end{equation*}
Clearly, $\widehat{\varGamma}\subset\varGamma^\prime$ and $\mu(\widehat{\varGamma})=1$.
By \textbf{(b)}, the sequence of functions $\left\{\frac{1}{k}\sum_{j=0}^{k-1}\frac{1}{t_\ell}\varphi(t_\ell,\theta^{jt_\ell}(\cdot))\right\}_{k=1}^{+\infty}$ is bounded below by an $\mu$-integrable function.
Thus from Fatou's lemma, there follows that
\begin{align*}
\int_{\widehat{\varGamma}}h_\ell^*(x)\mu(dx)&\le\liminf_{k\to+\infty}\int_{\widehat{\varGamma}}\frac{1}{k}\sum_{j=0}^{k-1}\left\{\frac{1}{t_\ell}\varphi(t_\ell,\theta^{jt_\ell}(x))-\bar{\varphi}(x)\right\}\mu(dx)\\
&\le\frac{\varepsilon}{2}
\end{align*}
for all $\ell\ge\ell(\varepsilon)$, and hence
\begin{equation*}
\lim_{\ell\to+\infty}\int_{\widehat{\varGamma}}h_\ell^*(x)\mu(dx)=0.
\end{equation*}
So, one can find a subsequence $\left\{h_{\ell_k}^*(\cdot)\right\}_{k=1}^{+\infty}$ such that
\begin{equation*}
h_{\ell_k}^*(x)\to 0\quad \textrm{as }k\to+\infty\qquad \textrm{for }
\mu\textrm{-a.e. }x\in\widehat{\varGamma}.
\end{equation*}
In addition, noting $t_{\ell+1}=2t_\ell$ for any $\ell\ge1$, for all $x\in\widehat{\varGamma}$
\begin{align*}
h_\ell^*(x)&=\lim_{k\to+\infty}\frac{1}{2k}\sum_{j=0}^{2k-1}\left\{\frac{1}{t_\ell}\varphi(t_\ell,\theta^{jt_\ell}(x))-\bar{\varphi}(x)\right\}\\
&\ge\lim_{k\to+\infty}\frac{1}{k}\sum_{j=0}^{k-1}\left\{\frac{1}{t_{\ell+1}}\varphi(t_{\ell+1},\theta^{jt_{\ell+1}}(x))-\bar{\varphi}(x)\right\}\\
&=h_{\ell+1}^*(x).
\end{align*}
This implies that
\begin{equation*}
h_{\ell}^*(x)\to 0\quad \textrm{as }\ell\to+\infty\quad \textrm{for
} \mu\textrm{-a.e. }x\in\widehat{\varGamma},
\end{equation*}
which proves the lemma.
\end{proof}

Now, we can readily prove Theorem~\ref{thm2.3} stated before based on Lemmas~\ref{lem2.4} and \ref{lem2.2} and Theorem~\ref{thm2.1}.

\begin{proof}[Proof of Theorem~\ref{thm2.3}]
Assume that $\mu$ is not supported on a periodic orbit of $f$; otherwise the statement holds trivially.
Let $\gamma,\gamma^\prime$ and $\gamma^{\prime\prime}$ be three constants such that
\begin{equation*}
\lambda_{\min}(\mu,f)>\gamma>\gamma^\prime>\gamma^{\prime\prime}>0,
\end{equation*}
and define the subadditive functions
\begin{equation*}
\varphi(t,x)=-\log\|D_{\!x}f^t\|_{\mathrm{co}}\quad\forall (t,x)\in\mathbb{N}\times M^n.
\end{equation*}
From Lemma~\ref{lem2.4} with $X=M^n$ and $\theta=f$, it follows that there is an $\ell>0$ and a non-periodic point $\by\in\mathrm{supp}(\mu)\cap\widehat{\varGamma}$, where $\widehat{\varGamma}$ is defined by Lemma~\ref{lem2.4}, such that
\begin{equation*}
\lim_{k\to+\infty}\frac{1}{k}\sum_{j=0}^{k-1}\log\|D_{\!{f^{jt_\ell}}(\by)}f^{t_\ell}\|_{\mathrm{co}}>\gamma{t_\ell}
\quad\textrm{and}\quad
\mu=\lim_{J\to+\infty}\frac{1}{J}\sum_{j=0}^{J-1}\delta_{f^j(\by)},
\end{equation*}
where $\delta_y$ denotes the Dirac measure at $y$ and the integer $t_\ell$ is given as in Lemma~\ref{lem2.4}.

From Lemma~\ref{lem2.2} with $a_j=\log\|D_{\!{f^{jt_\ell}}(\by)}f^{t_\ell}\|_{\mathrm{co}}$ for all $j\ge0$, it follows that one can  find a positive integer sequence $\{n_k\}_{k=1}^{+\infty}$ with $n_k\to+\infty$ such that
$(a_0,\ldots,a_{n_k-1})$ is a ``$\gamma^\prime t_\ell$-quasi-expanding
string" for all $k\ge1$, i.e.,
\begin{equation*}
\frac{1}{m}\sum_{j=1}^{m}a_{n_k-j}\geq\gamma^{\prime}t_\ell\qquad
\forall  m=1,\ldots,n_k.
\end{equation*}
For the simplicity of notation, we assume $t_\ell=1$; otherwise we consider $f^{t_\ell}$ instead of $f$ when applying Theorem~\ref{thm2.1}.

Write $x_j=f^j(\by)$ for all $j\ge0$. By the compactness of $M^n$, there can be found two subsequences $\{n_k^\prime\}_{k=1}^{+\infty}$ and $\{n_k^{\prime\prime}\}_{k=1}^{+\infty}$ of $\{n_k\}_{k=1}^{+\infty}$ such that
\begin{equation*}
\tau_k:=n_k^{\prime\prime}-n_k^\prime\to+\infty,\ \textsl{d}(x_{n_k^\prime}, x_{n_k^{\prime\prime}})\to 0\quad \textrm{as }k\to+\infty
\end{equation*}
and
\begin{equation*}
\frac{n_k^\prime}{\tau_k}\le\frac{1}{2}\quad \forall k\ge1.
\end{equation*}
Since
\begin{equation*}
\mu_k:=\frac{1}{\tau_k}\sum_{j=0}^{\tau_k-1}\delta_{x_{n_k^\prime+j}}=\left(1+\frac{n_k^\prime}{\tau_k}\right)\frac{1}{n_k^{\prime\prime}}\sum_{j=0}^{n_k^{\prime\prime}-1}\delta_{x_j}
-\frac{n_k^\prime}{\tau_k}\frac{1}{n_k^{\prime}}\sum_{j=0}^{n_k^{\prime}-1}\delta_{x_j},
\end{equation*}
there is no loss of generality in assuming that $\mu_k$ converges weakly-$*$ to $\mu$ as $k\to+\infty$. In fact, from
\begin{equation*}
\frac{1}{n_k^{\prime\prime}}\sum_{j=0}^{n_k^{\prime\prime}-1}\delta_{x_j}\xrightarrow[]{\textrm{weakly-}*}\mu\quad \textrm{and}\quad \frac{1}{n_k^{\prime}}\sum_{j=0}^{n_k^{\prime}-1}\delta_{x_j}\xrightarrow[]{\textrm{weakly-}*}\mu\quad \textrm{as }k\to+\infty,
\end{equation*}
it follows that
\begin{equation*}
\lim_{k\to+\infty}\int_{M^n}h(x)\mu_k(dx)=\lim_{k\to+\infty}\left(1+\frac{n_k^\prime}{\tau_k}-\frac{n_k^\prime}{\tau_k}\right)\int_{M^n}h(x)\mu(dx)
=\int_{M^n}h(x)\mu(dx)
\end{equation*}
for any $h\in\mathrm{C}^0(M^n)$.
Noting that $(x_{n_k^\prime}, f^{\tau_k}(x_{n_k^\prime}))$ is a $\gamma^\prime$-quasi-expanding orbit-string of $f$ with $\textsl{d}(x_{n_k^\prime}, f^{\tau_k}(x_{n_k^\prime}))$ converging to $0$, and $x_{n_k^{\prime\prime}}=f^{\tau_k}(x_{n_k^\prime})$.

Let $d_H(\cdot,\cdot)$ denote the Hausdorff metric for nonempty compact sets of $M^n$. Then, it follows, from Theorem~\ref{thm2.1} with $\lambda=\gamma^\prime$ and $0<\varepsilon<\gamma^\prime-\gamma^{\prime\prime}$, that there can be found a sequence of periodic repellers $\{P_{k_j}\}$ of $f$ with 
\begin{equation*}
\lim_{j\to+\infty}P_{k_j}=\Delta\subseteq\mathrm{supp}(\mu)\quad \textrm{and}\quad \lim_{j\to+\infty}d_H(P_{k_j}, (x_{n_{k_j}^\prime}, f^{\tau_{k_j}}(x_{n_{k_j}^\prime})))=0, 
\end{equation*}
and $f$ is nonuniformly expanding on $\bigcup_{k_j} P_{k_j}$ with an expansion indicator $\lambda\ge\gamma^{\prime\prime}$. So, we can obtain that $\lim_{j\to+\infty}\mathrm{supp}(\mu_{k_j})=\Delta$.

It is clear that $\Delta=\mathrm{supp}(\mu)$. In fact, if this fails, there is some $\hat{x}\in\mathrm{supp}(\mu)\setminus\Delta$ and further let $\textsl{d}(\hat{x},\Delta)=r>0$. Then there is a continuous function $\xi\colon M^n\rightarrow[0,1]$ satisfying
\begin{equation*}
\xi(\hat{x})=1\quad\textrm{and}\quad\xi(x)=0\ \forall x\in M^n\textrm{ with }\textsl{d}(x,\Delta)\le\frac{r}{2};
\end{equation*}
hence
\begin{equation*}
0<\int_{M^n}\xi(x)\mu(dx)=\lim_{j\to+\infty}\int_{M^n}\xi(x)\mu_{k_j}(dx)=\lim_{j\to+\infty}\int_{\mathrm{supp}(\mu_{k_j})}\xi(x)\mu_{k_j}(dx)=0,
\end{equation*}
a contradiction.

This ends the proof of Theorem~\ref{thm2.3}.
\end{proof}

\section{Natural extension of local diffeomorphisms}\label{sec3}%

For a diffeomorphism $f\colon M^n\rightarrow M^n$, Ma\~{n}\'{e}'s
arguments in~\cite{Man87} relies on the concept ``$(t,\gamma)$-set".
Let $K$ be an $f$-invariant compact subset of $M^n$ and
$t\ge 1, \gamma>0$. We say $K$ to be a $(t,\gamma)$-set of
$f$ if for every $z\in K$ there exists an integer ${\mathfrak
m}_{z}\in[0,t)$ such that
\begin{equation*}
\frac{1}{\ell}\sum_{i=1}^{\ell}\log\|D_{\!f^{-i}(f^{{\mathfrak
m}_{z}}(z))}f\|_{\mathrm{co}}\ge\gamma\qquad\forall \ell\ge1.
\end{equation*}
Now, since $f$ is not invertible, $f^{-i}$ makes no sense here. To
control the non-invertibility of $f$ in the context of
Theorem~\ref{thm1}, we need to introduce the natural extension of
$f$.

\subsection{The extension of a cocycle}\label{sec3.1}%

Let $f\colon X\rightarrow X$ be an arbitrarily given continuous
endomorphism of a compact metric space $(X,\textsl{d})$. Let
\begin{equation*}
\varSigma_{\!f}=\left\{\bar{x}=(\ldots, x_i,\ldots, x_{-1},x_0)\in
X^{\mathbb{Z}_-}\,|\,f(x_{i-1})=x_i\
\forall i\in\mathbb{Z}_-\right\},
\end{equation*}
where $\mathbb{Z}_-=\{0,-1,-2,\ldots\}$. Then,
\begin{equation*}
\textsl{d}_{\!f}(\bar{x},\bar{y})=\sum_{i=0}^{-\infty}2^i\min\{1,\textsl{d}(x_i,y_i)\}\qquad\forall \bar{x},\bar{y}\in\varSigma_{\!f}
\end{equation*}
is a metric on $\varSigma_{\!f}$ under which the shift mapping
\begin{equation*}
\sigma_{\!f}\colon\varSigma_{\!f}\rightarrow\varSigma_{\!f};\quad(\ldots,
x_i,\ldots, x_{-1},x_0)\mapsto(\ldots, x_i,\ldots,
x_{-1},x_0,f(x_0))
\end{equation*}
is a topological dynamical system (homeomorphism) on the compact metric space
$(\varSigma_{\!f},\textsl{d}_{\!f})$. Let
\begin{equation*}
\pi\colon\varSigma_{\!f}\rightarrow X;\quad (\ldots, x_i,\ldots,
x_{-1},x_0)\mapsto x_0
\end{equation*}
be the natural projection.

If $f$ is a $\mathrm{C}^1$-class local diffeomorphism of the closed manifold
$X=M^n$, we further set
\begin{equation*}
T_{\!\bar{x}}\varSigma_{\!f}=T_{\!\pi(\bar{x})}M^n\quad
\textrm{and}\quad F_{\bar{x}}\colon
T_{\!\bar{x}}\varSigma_{\!f}\rightarrow
T_{\!\sigma_{\!f}(\bar{x})}\varSigma_{\!f};\
v\mapsto(D_{\!\pi(\bar{x})}f)v\qquad\forall \bar{x}\in\varSigma_{\!f}.
\end{equation*}
Then, we obtain the natural linear skew-product dynamical system
\begin{equation*}
\begin{CD}
T\varSigma_{\!f}&@>{F}>>&T\varSigma_{\!f}\\
@V{\Pr}VV&{}&@VV{\Pr}V\\
\varSigma_{\!f}&@>{\sigma_{\!f}}>>&\varSigma_{\!f}
\end{CD}\quad \textrm{ where }
F(\bar{x},v)=(\sigma_{\!f}(\bar{x}), F_{\bar{x}}(v))\textrm{ and }
\Pr\colon(\bar{x},v)\mapsto\bar{x}.
\end{equation*}
Note here that the inverse of $\sigma_{\!f}$ is defined as
\begin{equation*}
\sigma_{\!f}^{-1}\colon\varSigma_{\!f}\rightarrow\varSigma_{\!f};\quad(\ldots,
x_i,\ldots, x_{-1},x_0)\mapsto(\ldots, x_i,\ldots, x_{-1}).
\end{equation*}

For any forward invariant set $\Lambda$ of $f$, let
$\Lambda_{\!f}=\pi^{-1}(\Lambda)$, which is called the ``extension'' of
$\Lambda$ under $f$. It is easy to see that $\Lambda_{\!f}$ is also a
forward $\sigma_{\!f}$-invariant set, i.e.,
$\sigma_{\!f}(\Lambda_{\!f})\subseteq\Lambda_{\!f}$.

The closure of a subset $Y$ in a topological space $Z$ is denoted by $\mathrm{Cl}_{Z}(Y)$. By the continuity of $\pi$ and the definition of $\textsl{d}_f(\cdot,\cdot)$, we can obtain that
\begin{equation*}
\pi(\mathrm{Cl}_{\varSigma_{\!f}}(\Lambda_{\!f}))\subseteq\mathrm{Cl}_{X}(\Lambda)\quad \textrm{and}\quad\pi(\Lambda_{\!f})=\Lambda
\end{equation*}
for any subset $\Lambda\subseteq X$.

We will need the following lemma.

\begin{lemma}\label{lem3.1}
Let $f\colon X\rightarrow X$ be a continuous endomorphism of a
compact metric space $X$. Then, for any forward $f$-invariant set
$\Lambda\subseteq X$, there follows
\begin{equation*}
\pi(\mathrm{Cl}_{\varSigma_{\!f}}(\Lambda_{\!f}))=\mathrm{Cl}_{X}(\Lambda),
\end{equation*}
where
$\pi\colon\varSigma_{\!f}\rightarrow X$ is the natural projection.
\end{lemma}

\begin{proof}
Let $\Lambda$ be a forward $f$-invariant subset of $X$. The statement trivially holds when $\Lambda=\varnothing$ or $X$. So, we now assume $\Lambda\not=\varnothing$ and $\not=X$. Let
$x\in\mathrm{Cl}_{X}(\Lambda)\setminus\Lambda$ be arbitrarily given. Then, there
is a sequence of points $p_j^{}$ in $\Lambda$ such that
\begin{equation*}
p_j^{}\to x\qquad\textrm{as }j\to+\infty.
\end{equation*}
For all $j=1,2,\ldots$, we arbitrarily pick
\begin{equation*}
\bar{p}_j^{}=(\ldots, p_{j,i}^{},\ldots, p_{j,-1}^{},p_{j,0}^{})\in\pi^{-1}(p_j^{})\subset\Lambda_{\!f}.
\end{equation*}
We now will define an $\bar{x}=(\ldots, x_{-1},
x_0)=(x_i)_{i=0}^{-\infty}\in\mathrm{Cl}_{\varSigma_{\!f}}(\Lambda_{\!f})$ with $\pi(\bar{x})=x$ by induction
on $i$ as follows:

First, let us choose $x_0=x$. Obviously, $p_{j,0}^{}\to x_0$ as $j\to+\infty$.

Secondly, since $X$ is compact, we can pick a convergent subsequence
$\left\{p_{j_k^{(1)},-1}^{}\right\}_{k=1}^{+\infty}$ from
$\left\{p_{j,-1}^{}\right\}_{j=1}^{+\infty}$. Let
$x_{-1}=\lim\limits_{k\to+\infty}p_{j_k^{(1)},-1}^{}$. It is easy to
see that $f(x_{-1})=x_0$ by the continuity of $f$. Assume $x_i, i\le -2$,
$\left\{j_k^{(-i)}\right\}_{k=1}^{\infty}$ have been defined such
that
\begin{equation*}
f(x_i)=x_{i+1}\quad \textrm{and}\quad p_{j_k^{(-i)},i}\to x_i\quad \textrm{as }
k\to+\infty.
\end{equation*}
By the compactness of the space $X$ once again, we can pick a convergent
subsequence, say $\left\{p_{j_k^{(-i+1)},i-1}^{}\right\}_{k=1}^{+\infty}$,
from $\left\{p_{j,i-1}^{}\right\}_{j=1}^{+\infty}$ such that
$\left\{j_k^{(-i+1)}\right\}_{k=1}^\infty\subset\left\{j_k^{(-i)}\right\}_{k=1}^\infty$.
We let $x_{i-1}=\lim\limits_{k\to\infty}p_{j_k^{(-i+1)},i-1}^{}$. Then from
the construction, we easily get $f(x_{i-1})=x_i$. This completes the
induction step.

Thus, we have chosen a point $\bar{x}=(\ldots, x_i,\ldots,
x_1,x_0)\in\varSigma_{\!f}$ such that $\pi(\bar{x})=x$ and
$\bar{x}\in\mathrm{Cl}_{\varSigma_{\!f}}(\Lambda_{\!f})$. This implies that $\pi(\mathrm{Cl}_{\varSigma_{\!f}}(\Lambda_{\!f}))\supseteq\mathrm{Cl}_{X}(\Lambda)$.

This shows Lemma~\ref{lem3.1}.
\end{proof}

We note that generally $\mathrm{Cl}_{X}(\Lambda)_{\!f}$ is
bigger than $\mathrm{Cl}_{\varSigma_{\!f}}(\Lambda_{\!f})$ when $f$ is not injective, for an arbitrary $f$-invariant set $\Lambda\subset X$.
\subsection{Obstruction and $(t,\gamma)$-set}\label{sec3.2}%

Hereafter, let $f\colon M^n\rightarrow M^n$ be an arbitrarily given
$\mathrm{C}^1$-class local diffeomorphism of the closed manifold $M^n$. Let
$\sigma_{\!f}\colon\varSigma_{\!f}\rightarrow\varSigma_{\!f}$ and $F\colon
T\varSigma_{\!f}\rightarrow T\varSigma_{\!f}$ be the natural
extensions of $f$ defined as in Section~\ref{sec3.1}.

\begin{defn}\label{def3.2}
For $\bar{x}\in\varSigma_f$ and $m\ge1$, $(\bar{x},\sigma_{\!f}^m(\bar{x}))$ is called a {\it $\rho$-string}
of $F$ if
\begin{equation*}
\frac{1}{m}\sum_{i=0}^{m-1}\log\|F_{\sigma_{\!f}^i(\bar{x})}\|_{\mathrm{co}}\ge\rho.
\end{equation*}
Given any $\mathfrak n\ge 1$ and $\varrho>0$, we say
$(\bar{x},\sigma_{\!f}^m(\bar{x}))$ is an {\it $(\mathfrak
n,\varrho)$-obstruction} of $F$ if $m\ge\mathfrak n$ and
$(\bar{x},\sigma_{\!f}^\ell(\bar{x}))$ is not a
$\varrho$-string of $F$ for all $\mathfrak n\le \ell\le m$,
i.e.,
\begin{equation*}
\frac{1}{\ell}\sum_{i=0}^{\ell-1}\log\|F_{\!\sigma_{\!f}^i(\bar{x})}\|_{\mathrm{co}}<\varrho\qquad\forall \ell\in[\mathfrak
n,m].
\end{equation*}
Note here that
$F_{\!\sigma_{\!f}^i(\bar{x})}=D_{\!f^i(\pi(\bar{x}))}f$ for any
$i\ge 0$ from the definition of $F$ in Section~\ref{sec3.1}.
\end{defn}

It is easily seen that if $F$ is not expanding on a forward
$\sigma_{\!f}$-invariant closed set $\Theta\subset\varSigma_{\!f}$, then to any
$\mathfrak n\ge 1$ and $\varrho>0$, from the compactness of
$\varSigma_{\!f}$ there can be found at least one $\bar{x}\in\Theta$ such that
$(\bar{x},\sigma_{\!f}^m(\bar{x}))$ is an $(\mathfrak
n,\varrho)$-obstruction of $F$ for all $m\ge\mathfrak n$.

\begin{lemma}[{\cite[Lemma~II-4]{Man87}}]\label{lem3.3}
Let $\bar{\gamma}_2^{}>\gamma_3^{}>0$ and
$(\bar{x},\sigma_{\!f}^m(\bar{x}))$ be a $\bar{\gamma}_2^{}$-string
of $F$, i.e.,
\begin{equation*}
\frac{1}{m}\sum_{i=0}^{m-1}\log\|F_{\sigma_{\!f}^i(\bar{x})}\|_{\mathrm{co}}\ge\bar{\gamma}_2.
\end{equation*}
Let $0<n_1<\cdots<n_k\le m$ be the set of integers such that
$(\bar{x},\sigma_{\!f}^{n_i}(\bar{x}))$ is a
$\gamma_3^{}$-quasi-expanding string of $F$, i.e.,
\begin{equation*}
\frac{1}{\ell}\sum_{j=1}^{\ell}\log\|F_{\sigma_{\!f}^{n_i-j}(\bar{x})}\|_{\mathrm{co}}\ge{\gamma}_3\quad\forall \ell=1,\ldots,n_i\textrm{ for }1\le i\le k.
\end{equation*}
Then, for all $1\le
i<k$, either $n_{i+1}-n_i\le N_{\bar{\gamma}_2^{},\gamma_3^{}}$ or
$(\sigma_{\!f}^{n_i}(\bar{x}), \sigma_{\!f}^{n_{i+1}}(\bar{x}))$ is
an
$(N_{\bar{\gamma}_2^{},\gamma_3^{}},\bar{\gamma}_2^{})$-obstruction
of $F$. Here $N_{\bar{\gamma}_2^{},\gamma_3^{}}$ is defined in the
manner as in Lemma~\ref{lem2.2} with constants $\gamma=\bar{\gamma}_2$, $\gamma^\prime=\gamma_3$ and
$H=\max\left\{|\log\|D_{\!x}f\|_{\mathrm{co}}|;\ x\in
M^n\right\}$.
\end{lemma}

The following result is a special case of \cite[Lemma~II-5]{Man87}
in the case of $r=0$.

\begin{lemma}[{\cite[Lemma~II-5]{Man87}}]\label{lem3.4}
Let there be any given real numbers
\begin{equation*}
\gamma_0^{}>\gamma_1^{}>\gamma_2^{}>\gamma_3^{}>0
\end{equation*}
and integers
\begin{equation*}
m>\ell>\mathfrak n>0.
\end{equation*}
Let $(\bar{x},\sigma_{\!f}^m(\bar{x}))$ be a
$\gamma_0^{}$-string of $F$ and
$(\bar{x},\sigma_{\!f}^\ell(\bar{x}))$ an $(\mathfrak
n,\gamma_2^{})$-obstruction of $F$. Assume
\begin{description}
\item[(a)] $m\ge N_{\gamma_0^{},\gamma_3^{}}$,

\item[(b)] $mc_{\gamma_0^{},\gamma_3^{}}>\ell$,

\item[(c)] $\ell\ge N_{\gamma_1^{},\gamma_2^{}}$ and

\item[(d)] $\ell c_{\gamma_1^{},\gamma_2^{}}>\mathfrak n$.
\end{description}
Then, there exists a $\gamma_3^{}$-quasi-expanding string
$(\bar{x},\sigma_{\!f}^k(\bar{x}))$ of $F$ with $\ell\le k<m$, such
that $(\bar{x},\sigma_{\!f}^k(\bar{x}))$ is not a
$\gamma_1^{}$-string of $F$. Here all $N_{\gamma_0^{},\gamma_3^{}},
c_{\gamma_0^{},\gamma_3^{}}$ and $N_{\gamma_1^{},\gamma_2^{}},
c_{\gamma_1^{},\gamma_2^{}}$ are defined in the manner as in
Lemma~\ref{lem2.2} with $H=\max\left\{|\log\|D_{\!x}f\|_{\mathrm{co}}|;\ x\in
M^n\right\}$.
\end{lemma}

Let $\Theta\subset\varSigma_{\!f}$ be a forward invariant non-void
closed set of $\sigma_{\!f}$ and $\bar{x}\in\Theta$. Following \cite{Man87}, we define the ``germ" as follows:
\begin{equation*}
\mathrm{J}(\bar{x},\Theta)=\left\{\bar{y}\in\Theta\,|\,\exists\{x_k\}\textrm{ and }\{m_k\}\textrm{ such that }\bar{y}=\lim_{k\to+\infty}\sigma_{\!f}^{m_k}(\bar{x}_k)\right\}
\end{equation*}
where $\{\bar{x}_k\}$ is a sequence in $\Theta$ converging
to $\bar{x}$ and $m_k\to+\infty$.

Clearly to obtain $\mathrm{J}(\bar{x},\Theta)$, it is sufficient to
use sequences $\{\bar{x}_k\}$ contained in a dense subset $\Theta_0$
of $\Theta$. Since $m_k$ converges to $+\infty$ as $k\to+\infty$, it is
easily seen that $\mathrm{J}(\bar{x},\Theta)$ is closed and
$\sigma_{\!f}$-invariant. Moreover, if
$\Theta=\Omega(\sigma_{\!f}|\Theta)$ the nonwandering point set of the restriction of $\sigma_{\!f}$ to $\Theta$, then
$\bar{x}$ itself belongs to $\mathrm{J}(\bar{x},\Theta)$. In addition,
$\mathrm{J}(\bar{x},\Theta)$ is nonempty because every
$\omega$-limit point of $\bar{x}$ belongs to it.

The following notion is a modification of Ma\~{n}\'{e}'s
``$(t,\gamma)$-set'' defined there for a diffeomorphism~\cite[pp.
177]{Man87}.

\begin{defn}\label{def3.5}
A forward $\sigma_{\!f}$-invariant compact subset $\varSigma^\prime$ of
$\varSigma_{\!f}$ is called a {\it $(t,\gamma)$-set} of $F$ where
$t\in\mathbb{N}$ and $\gamma>0$, provided that for every
$\bar{z}\in\varSigma^\prime$, there exists an integer ${\mathfrak
m}_{\bar{z}}\in[0,t)$ such that $\left(\sigma_{\!f}^{\mathfrak
m_{\bar{z}}-r}(\bar{z}), \sigma_{\!f}^r\left(\sigma_{\!f}^{\mathfrak
m_{\bar{z}}-r}(\bar{z})\right)\right)$ is a $\gamma$-string of $F$
for all $r>0$.
\end{defn}

Clearly, this implies that $F$ is expanding on $\varSigma^\prime$. Note
here that we do not require $\sigma_{\!f}^{\mathfrak
m_{\bar{z}}-r}(\bar{z})$ to be in $\varSigma^\prime$, since $\varSigma^\prime$ is only forwardly invariant.

The following lemma will be needed in the proof of our
Theorem~\ref{thm1}.

\begin{lemma}\label{lem3.6}
Let there be given a forward $\sigma_{\!f}$-invariant compact set
$\Theta\subset\varSigma_{\!f}$, on which $F$ is not expanding.
Assume there is the number $c>0$ for which there holds the
inequality
\begin{equation*}
\limsup_{k\to+\infty}\frac{1}{k}\sum_{i=0}^{k-1}\log\|F_{\!\sigma_{\!f}^i(\bar{x})}\|_{\mathrm{co}}\ge
c
\end{equation*}
for a dense set $\Theta_{0}$ of points $\bar{x}\in\Theta$. Then, for
any $\epsilon>0$ and
$c>\gamma_2^{}>\bar{\gamma}_2^{}>\gamma_3^{}>0$, there exists an
integer $N=N_{\epsilon;
\gamma_2^{},\bar{\gamma}_2^{},\gamma_3^{}}\ge 1$ such that for any
$\bar{x}\in\Theta$,
\begin{description}
\item[(1)] either $\mathrm{J}(\bar{x},\Theta)$ is an
$(N,\gamma_3^{})$-set of $F$;

\item[(2)] or there exists $\bar{y}\in\Theta$ such
that $(\bar{y},\sigma_{\!f}^m(\bar{y}))$ is an
$(N_{\bar{\gamma}_2^{},\gamma_3^{}},\gamma_2^{})$-obstruction of $F$
for all $m\ge N_{\bar{\gamma}_2^{},\gamma_3^{}}$, where
$N_{\bar{\gamma}_2^{},\gamma_3^{}}$ is given by Lemma~\ref{lem2.2} with $H=\max_{x\in
M^n}|\log\|D_{\!x}f\|_{\mathrm{co}}|$;
moreover, such $\bar{y}$ satisfies at least one of the following
properties:
\begin{description}
\item[a)] $\textsl{d}_{\!f}(\bar{x},\bar{y})\le\epsilon$;

\item[b)] there exists $\bar{u}\in\Theta$ arbitrarily near to $\bar{x}$ and
$m\ge 1$ such that
$\textsl{d}_{\!f}(\sigma_{\!f}^m(\bar{u}),\bar{y})<\epsilon$ and that
$(\bar{u},\sigma_{\!f}^m(\bar{u}))$ is a
$\gamma_3^{}$-quasi-expanding string of $F$.
\end{description}
\end{description}
\end{lemma}

\begin{proof}
The argument is a modification of that of
\cite[Lemma~II-6]{Man87}. We denote by $\Theta^\prime$ the set of
points $\bar{y}\in\Theta$ such that
$(\bar{y},\sigma_{\!f}^m(\bar{y}))$ is an
$(N_{\bar{\gamma}_2^{},\gamma_3^{}},\gamma_2^{})$-obstruction of $F$
for all $m\ge N_{\bar{\gamma}_2^{},\gamma_3^{}}$. Since $F$ is not
expanding on $\Theta$ by hypothesis, the set $\Theta^\prime$ is
obviously non-void.

It is easy to check that there exists $N=N_{\epsilon;
\gamma_2^{},\bar{\gamma}_2^{},\gamma_3^{}}>N_{\bar{\gamma}_2^{},\gamma_3^{}}$
such that when $(\bar{y},\sigma_{\!f}^m(\bar{y}))$ is an
$(N_{\bar{\gamma}_2^{},\gamma_3^{}},\bar{\gamma}_2^{})$-obstruction
of $F$ in $\Theta$ for some $m>N$, then
$\textsl{d}_{\!f}(\bar{y},\Theta^\prime)<\epsilon$. Here we have used the
hypothesis $\gamma_2^{}>\bar{\gamma}_2^{}$, the $\mathrm{C}^1$-smoothness of
$f$, and the compactness of $\Theta$.

Given any $\bar{x}\in\Theta$ and any
$\bar{z}\in\mathrm{J}(\bar{x},\Theta)$, there exists a
sequence $\{\bar{x}_k\}_{k\ge 1}$ in $\Theta_0$ converging to
$\bar{x}$ and satisfying
$\bar{z}=\lim\limits_{k\to+\infty}\sigma_{\!f}^{m_k}(\bar{x}_k)$,
where $m_k$ converges to $+\infty$ as $k$ tends to $+\infty$. For any $k\ge 1$, define the set of integers
\begin{equation*}
\mathscr{S}_k=\left\{m>0\,|\,\left(\bar{x}_k,\sigma_{\!f}^m(\bar{x}_k)\right)\textrm{
is a } \gamma_3^{}\textrm{-quasi-expanding string of
$F$}\right\}\cup\{0\}.
\end{equation*}
As $\gamma_3^{}<c$ and $\bar{x}_k\in\Theta_0$, from
Lemma~\ref{lem2.2} it follows easily that $\mathscr{S}_k$ is
infinite. For any $k\ge 1$, set
\begin{equation*}
k^+=\min\mathscr{S}_k\cap[m_k,+\infty)\quad \textrm{and}\quad
k^-=\max\mathscr{S}_k\cap [0,m_k).
\end{equation*}

Suppose that $\liminf\limits_{k\to+\infty}(k^+-k^-)\le N$. Then,
there exists some integer $\mathfrak m_{\bar{z}}\in[0,N]$ which
satisfies that $\sigma_{\!f}^{\mathfrak m_{\bar{z}}}(\bar{z})$ is
the limit of some subsequence of
$\left\{\sigma_{\!f}^{k^+}(\bar{x}_k)\,|\,k\ge 1\right\}$. Hence,
$\left(\sigma_{\!f}^{\mathfrak
m_{\bar{z}}-r}(\bar{z}),\sigma_{\!f}^{\mathfrak
m_{\bar{z}}}(\bar{z})\right)$ is a $\gamma_3^{}$-string of $F$ for
all $r\ge 1$ (here we use the property $m_k\to+\infty$).
If this holds for all $\bar{z}\in\mathrm{J}(\bar{x},\Theta)$ then,
$\mathrm{J}(\bar{x},\Theta)$ is an $(N,\gamma_3^{})$-set of $F$. If
this is not the case, the above argument shows that we can always
pick some $\bar{z}\in \mathrm{J}(\bar{x},\Theta)$ such that for any
sufficiently large $k$, $k^+-k^->N$. Hence
$k^+-k^->N_{\bar{\gamma}_2^{},\gamma_3^{}}$ because of
$N>N_{\bar{\gamma}_2^{},\gamma_3^{}}$. Then, by Lemmas~\ref{lem2.2}
and~\ref{lem3.3}, it follows that
$\left(\sigma_{\!f}^{k^-}(\bar{x}_k),\sigma_{\!f}^{k^+}(\bar{x}_k)\right)$
is an
$(N_{\bar{\gamma}_2^{},\gamma_3^{}},\bar{\gamma}_2^{})$-obstruction
of $F$ in $\Theta$. So,
$\textsl{d}_{\!f}\left(\sigma_{\!f}^{k^-}(\bar{x}_k),\Theta^\prime\right)<\epsilon$
for sufficiently large $k$. If $k^-=0$ for all sufficiently large
$k$ that satisfy
$\textsl{d}_{\!f}\left(\sigma_{\!f}^{k^-}(\bar{x}_k),\Theta^\prime\right)<\epsilon$,
then $\textsl{d}_{\!f}(\bar{x}_k,\Theta^\prime)<\epsilon$ and since
$\bar{x}_k\to \bar{x}$ we obtain
$\textsl{d}_{\!f}(\bar{x},\Theta^\prime)\le\epsilon$. Taking
$\bar{y}\in\Theta^\prime$ such that
$\textsl{d}_{\!f}(\bar{x},\bar{y})\le\epsilon$, it follows that
$\bar{y}$ satisfies Lemma~\ref{lem3.6} and meanwhile the stipulation
\textbf{a)}. On the other hand, if for an unbounded set of $k$ we have
$k^->0$, we can take $\bar{y}\in\Theta^\prime$ such that
$\textsl{d}_{\!f}\left(\sigma_{\!f}^{k^-}(\bar{x}_k),\bar{y}\right)<\epsilon$
and then this point $\bar{y}$, the point $\bar{u}=\bar{x}_k$ and
$m=k^-$ satisfy the requirements of Lemma~\ref{lem3.6} and the item
\textbf{b)}.

This completes the proof of Lemma~\ref{lem3.6}.
\end{proof}
\section{A sifting-shadowing combination}\label{sec4}%

This section is devoted to the most important argument of the
sifting-shadowing combination for the proof of our main result
Theorem~\ref{thm1}.

By a sifting-shadowing combination, we will obtain the following
criterion, which is of independent intrinsic interest.

\begin{theorem}\label{thm4.1}
Let there be given the forward invariant subset $\Lambda\subset M^n$ of a
$\mathrm{C}^1$-class endomorphism $f\colon M^n\rightarrow M^n$, whose closure
$\mathrm{Cl}_{M^n}(\Lambda)$ in $M^n$ does not contain any critical points of $f$. Assume
there is the number $c>0$ for which there holds the inequality:
\begin{equation*}
\limsup_{k\to+\infty}\frac{1}{k}\sum_{i=0}^{k-1}\log\|D_{\!f^i(x)}f\|_{\mathrm{co}}\ge
c\qquad\forall x\in\Lambda.
\end{equation*}
If $\Lambda\subseteq\mathrm{Per}(f)$ then, either $f$ is expanding on
$\mathrm{Cl}_{M^n}(\Lambda)$ or for every neighborhood $V$ of $\Lambda$ in
$M^n$ and every $0<\gamma^\prime<\gamma^{\prime\prime}<c$, there can
be found in $V$ a periodic orbit $P$ of $f$ with arbitrarily large
period $\tau_{\!P}$ and satisfying the following ``abnormal
inequality" property:
\begin{equation*}
\frac{1}{\tau_{\!P}}\sum_{i=0}^{\tau_{\!P}-1}\log\|D_{\!f^i(p)}f\|_{\mathrm{co}}<\gamma^{\prime\prime}
\end{equation*}
and
\begin{equation*}
\frac{1}{k}\sum_{i=1}^{k}\log\|D_{\!f^{\tau_{\!P}-i}(p)}f\|_{\mathrm{co}}>\gamma^\prime\quad \forall k=1,
\ldots, \tau_{\!P},
\end{equation*}
for some point $p\in P$.
\end{theorem}

We are going to prove this theorem following the framework of the
proof of Ma\~{n}\'{e}~\cite[Theorem~II-1]{Man87} that was
clarified independently by \cite{MS,ZG}. Let
$\textsl{d}(\cdot,\cdot)$ be a metric on
$M^n$ compatible with the natural norm $\|\cdot\|$ on $T\!{M^n}$. The $\eta$-neighborhood of a set $X\subset M^n$, denoted by $B_\eta(X)$, is the union of the $\eta$-balls $B_\eta(x)$ around the points $x\in X$.

\begin{proof}
Since $f$ does not have any critical points in $\mathrm{Cl}_{M^n}(\Lambda)$ and it is
of class $\mathrm{C}^1$, $f$ is $\mathrm{C}^1$-class locally diffeomorphic restricted to a closed neighborhood of $\Lambda$. For simplicity, there is no loss of generality in assuming that $f$
is $\mathrm{C}^1$-class locally diffeomorphic on the whole manifold $M^n$.

Let $\Lambda\not=\varnothing$. If $f$ is expanding on
$\mathrm{Cl}_{M^n}(\Lambda)$, we may stop proving here. Now we assume $f$ is
not uniformly expanding on $\Lambda$.

To prove Theorem~\ref{thm4.1}, we consider the natural extensions
\begin{equation*}
\sigma_{\!f}\colon\varSigma_{\!f}\rightarrow\varSigma_{\!f}\quad
\textrm{and}\quad F\colon T\varSigma_{\!f}\rightarrow
T\varSigma_{\!f}
\end{equation*}
associated to $f$ defined as in Section~\ref{sec3.1}. Write
$\Theta_0=\Lambda_{\!f}$ and $\Theta=\mathrm{Cl}_{\varSigma_{\!f}}(\Lambda_{\!f})$,
where $\Lambda_{\!f}$ is the natural extension of $\Lambda$ under
$f$ defined in the way as in Section~\ref{sec3.1}. So, by the hypothesis
of Theorem~\ref{thm4.1}, $F$ is not expanding on $\Theta$ such that
\begin{equation*}
\limsup_{k\to+\infty}\frac{1}{k}\sum_{i=0}^{k-1}\log\|F_{\!\sigma_{\!f}^i(\bar{x})}\|_{\mathrm{co}}\ge
c\qquad\forall \bar{x}\in\Theta_0,
\end{equation*}
where $c$ is the constant given in the statement of
Theorem~\ref{thm4.1}. Moreover,
$\Theta=\Omega({\sigma_{\!f}|\Theta})$, the nonwandering set of $\sigma_{\!f}$ restricted to $\Theta$.

Let $0<\gamma^\prime<\gamma^{\prime\prime}<c$ be as in the statement
of Theorem~\ref{thm4.1}. Then, from now on we fix any $\gamma_0^{}$
with $\gamma^\prime<\gamma_0^{}<\gamma^{\prime\prime}$. To any
$\bar{x}\in\Theta_0$, there can be found a sequence of positive
integers $n_j(\bar{x})\uparrow+\infty$ satisfying
\begin{equation*}
\frac{1}{n_j(\bar{x})}\sum_{i=0}^{n_j(\bar{x})-1}\log\|F_{\!\sigma_{\!f}^i(\bar{x})}\|_{\mathrm{co}}>\gamma_0^{}\qquad
\forall j=1,2,\dotsc.
\end{equation*}
Let $\bar{\gamma}$ and $\hat{\gamma}$ be two arbitrary numbers
such that $\gamma_0^{}>\bar{\gamma}>\hat{\gamma}>\gamma^\prime$.
Choose $\eta\in(0,1)$ sufficiently small satisfying
\begin{equation*}
\hat{\gamma}-\eta>\gamma^\prime\quad\textrm{and}\quad\bar{\gamma}+\eta<\gamma^{\prime\prime}.
\end{equation*}
Take $\varepsilon>0$ so small that
\begin{equation*}
\hat{\gamma}-\varepsilon>\gamma^\prime,\quad
B_{2\varepsilon}(\Lambda)\subset V
\end{equation*}
and that if $\textsl{d}(x,y)\le\varepsilon$ for any $x,y\in M^n$
then
\begin{equation*}
\left|\log\|D_{\!x}f\|_{\mathrm{co}}-\log\|D_{\!y}f\|_{\mathrm{co}}\right|<\eta.
\end{equation*}
Let $\delta=\delta(\varepsilon,\hat{\gamma})$ be given as in the
statement of Theorem~\ref{thm2.1} with
$\lambda=\hat{\gamma}$.

From the compactness of $\varSigma_{\!f}$, we can choose the
positive integer $s=s(\delta/4)$ which satisfies that for any given
sequence $\{\bar{x}_1,\bar{x}_2,\ldots,\bar{x}_s\}$ in
$\varSigma_{\!f}$ there always can be found $i,j$ with $1\le
i\not=j\le s$ such that
$\textsl{d}_{\!f}(\bar{x}_i,\bar{x}_j)<\delta/4$. Now, we define
$4(s+1)$ positive numbers as follows:
\begin{equation*}
\left\{\uwave{\gamma_1^{(i)},
\gamma_2^{(i)},\bar{\gamma}_2^{(i)},\gamma_3^{(i)}}\colon i=0,1,\ldots,s\right\}
\end{equation*}
such that
$$
\begin{array}{rcl}
\hat{\gamma}=\uwave{\gamma_3^{(0)}<\bar{\gamma}_2^{(0)}<\gamma_2^{(0)}<\gamma_1^{(0)}}<\uwave{\gamma_3^{(1)}<\bar{\gamma}_2^{(1)}<\gamma_2^{(1)}<\gamma_1^{(1)}}<\cdots<
\gamma_1^{(i-1)}&\\
<\uwave{\gamma_3^{(i)}<\bar{\gamma}_2^{(i)}<\gamma_2^{(i)}<\gamma_1^{(i)}}<\gamma_3^{(i+1)}<\cdots<\gamma_1^{(s-1)}&\\
<\uwave{\gamma_3^{(s)}<\bar{\gamma}_2^{(s)}<\gamma_2^{(s)}<\gamma_1^{(s)}}&=\bar{\gamma}.
\end{array}
$$
Let
\begin{equation*}
\mathfrak
n_i=N_{\frac{\delta}{4};\gamma_2^{(i)},\bar{\gamma}_2^{(i)},\gamma_3^{(i)}}>N_{\bar{\gamma}_2^{(i)},\gamma_3^{(i)}}
\end{equation*}
be the constants determined by Lemma~\ref{lem3.6} in the case of
letting $\epsilon=\frac{\delta}{4}, \gamma_2^{}=\gamma_2^{(i)},
\bar{\gamma}_2^{}=\bar{\gamma}_2^{(i)}$ and
$\gamma_3^{}=\gamma_3^{(i)}$ for all $0\le i\le s$.

For any $i=0,\ldots,s$, let $\varSigma_i$ be the compact set which
consists of points $\bar{z}\in\Theta$ such that there exists an
integer $\mathfrak m_{\bar{z}}\in[0,\mathfrak n_i)$ verifying that
$\left(\sigma_{\!f}^{\mathfrak m_{\bar{z}}-r}(\bar{z}),
\sigma_{\!f}^r(\sigma_{\!f}^{\mathfrak
m_{\bar{z}}-r}(\bar{z}))\right)$ is a $\gamma_3^{(i)}$-string
of $F$ for all $r>0$, i.e.,
\begin{equation*}
\frac{1}{r}\sum_{i=0}^{r-1}\log\|F_{\!\sigma_{\!f}^i(\sigma_{\!f}^{\mathfrak
m_{\bar{z}}-r}(\bar{z}))}\|_{\mathrm{co}}\ge\gamma_3^{(i)}\qquad
\forall r>0.
\end{equation*}
It need not be $\sigma_{\!f}$-invariant. Clearly,
$\Theta\not=\bigcup_{i=0}^s\varSigma_i$, since $F$ is not expanding on
$\Theta$.

For any $0\le i\le s-1$, let us choose arbitrarily $\bar{x}_i\in\Theta\setminus\varSigma_i$. Then,
$\mathrm{J}(\bar{x}_i,\Theta)$ is not an $\left(\mathfrak
n_i,\gamma_3^{(i)}\right)$-set of $F$ because
$\bar{x}_i$ belongs to $\mathrm{J}(\bar{x}_i,\Theta)$. From Lemma~\ref{lem3.6},
there can be found $\bar{u}_{i}, \bar{y}_{i}\in\Theta$ and $m_i\ge
0$ such that $\textsl{d}_{\!f}(\bar{x}_i,\bar{u}_{i})<\delta/4$,
$\textsl{d}_{\!f}(\sigma_{\!f}^{m_i}(\bar{u}_{i}),\bar{y}_{i})<\delta/4$
and $\left(\bar{u}_{i}, \sigma_{\!f}^{m_i}(\bar{u}_{i})\right)$ is a
$\gamma_3^{(i)}$-quasi-expanding string of $F$ if $m_i>0$, and that
$\left(\bar{y}_i,\sigma_{\!f}^m(\bar{y}_i)\right)$ is an
$\left(N_{\bar{\gamma}_2^{(i)},\gamma_3^{(i)}},\gamma_2^{(i)}\right)$-obstruction
of $F$ for all $m>\mathfrak n_i$. As $\Theta_0$ is dense in
$\Theta$, it follows that when $\bar{z}_i\in\Theta_0$ is
sufficiently close to $\bar{y}_i$, there exists a large
$\ell>\mathfrak n_i$ such that
$\left(\bar{z}_i,\sigma_{\!f}^{\ell}(\bar{z}_i)\right)$ is an
$\left(N_{\bar{\gamma}_2^{(i)},\gamma_3^{(i)}},\gamma_2^{(i)}\right)$-obstruction
of $F$. Moreover, there exist infinitely many $m$ such that
\begin{equation*}
\frac{1}{m}\sum_{j=0}^{m-1}\log\|F_{\!\sigma_{\!f}^j(\bar{z}_i)}\|_{\mathrm{co}}>\gamma_0^{}.
\end{equation*}
Applying Lemma~\ref{lem3.4} to $\gamma_1^{}=\gamma_1^{(i)},
\gamma_2^{}=\gamma_2^{(i)}$ and $\gamma_3^{}=\gamma_3^{(i)}$,
because $m$ can be chosen large with respect to $\ell$, $\ell$ large
with respect to $\mathfrak n_i$, there exists $m>k_i>\ell>\mathfrak
n_i$ such that
$\left(\bar{z}_i,\sigma_{\!f}^{k_i}(\bar{z}_i)\right)$ is a
$\gamma_3^{(i)}$-quasi-expanding string of $F$ but not a
$\gamma_1^{(i)}$-string of $F$, and so not a
$\gamma_3^{(i+1)}$-string of $F$ too. Thus,
$\sigma_{\!f}^{k_i}(\bar{z}_i)$ lies in $\Theta\setminus\varSigma_{i+1}$ when
$\ell$ large enough. Further, by induction on $i$ we can construct
sequences $\left\{(\bar{u}_i,\bar{z}_i)\right\}_{i=0}^{s-1}$ and
$\left\{(m_i,k_i)\right\}_{i=0}^{s-1}$ with
$\bar{u}_i,\bar{z}_i\in\Theta$ and $m_i\ge 0,k_i\ge2$, such that:
\begin{description}
\item[1)] $\left(\bar{u}_i, \sigma_{\!f}^{m_i}(\bar{u}_i)\right)$ and $\left(\bar{z}_i,
\sigma_{\!f}^{k_i}(\bar{z}_i)\right)$ both are
$\gamma_3^{(i)}$-quasi-expanding string of $F$, where
$\bar{u}_i=\bar{z}_i$ if $m_i=0$;

\item[2)] $\left(\bar{z}_i, \sigma_{\!f}^{k_i}(\bar{z}_i)\right)$ is not a
$\gamma_1^{(i+1)}$-string of $F$;

\item[3)] $\textsl{d}_{\!f}\left(\sigma_{\!f}^{m_i}(\bar{u}_i), \bar{z}_i\right)<\delta/2$;

\item[4)] $\textsl{d}_{\!f}\left(\sigma_{\!f}^{k_i}(\bar{z}_i), \bar{u}_{i+1}\right)<\delta/2$;

\item[5)] if $H=\max\{\log\|D_{\!x}f\|_{\mathrm{co}};\ x\in M^n\}$,
then
\begin{equation*}
k_i\gamma_1^{i+1}+m_i H\le(m_i+k_i)\bar{\gamma}
\end{equation*}
for all $i=0,\dotsc,s-1$.
\end{description}
In fact, because $\bar{\gamma}>\gamma_1^{(i)}$ for all $i$, we only
need to take $k_i$ sufficiently large in the previous construction
to satisfy the conditions \textbf{1)}, \textbf{2)}, \textbf{3)}, \textbf{4)}, and \textbf{5)} above.

By the definition of $s=s(\delta/4)$ before, there can be found in
$\{\bar{u}_i\}_{i=0}^{s-1}$ two points $\bar{u}_i,\bar{u}_{j+1}$
with $i<j$ such that
$\textsl{d}_{\!f}(\bar{u}_i,\bar{u}_{j+1})<\delta/4$. It is easy to
check that the sequence
\begin{equation*}
{(\bar{u}_i, \sigma_{\!f}^{m_i}(\bar{u}_i)),
(\bar{z}_i, \sigma_{\!f}^{k_i}(\bar{z}_i))}, \dotsc,
{(\bar{u}_j, \sigma_{\!f}^{m_j}(\bar{u}_j)),
(\bar{z}_j, \sigma_{\!f}^{k_j}(\bar{z}_j))}
\end{equation*}
forms a periodic $\hat{\gamma}$-quasi-expanding of $F$
$\delta$-pseudo-orbit of $\sigma_{\!f}$ in $\Theta$. Let
\begin{equation*}
u_\ell=\pi(\bar{u}_\ell)\quad \textrm{and}\quad
z_\ell=\pi(\bar{z}_\ell)\qquad \textrm{for all }i\le\ell\le j,
\end{equation*}
where $\pi\colon\varSigma_{\!f}\rightarrow M^n$ is the natural
projection defined as in Section~\ref{sec3.1}. Then, by the definition
of the metric function $\textsl{d}_{\!f}$ of $\varSigma_{\!f}$ it
follows from Lemma~\ref{lem3.1} that the string
\begin{equation*}
{({u}_i, f^{m_i}({u}_i)), ({z}_i,
f^{k_i}({z}_i))}, \dotsc, {({u}_j,
f^{m_j}({u}_j)), ({z}_j, f^{k_j}({z}_j))}
\end{equation*}
forms a periodic $\hat{\gamma}$-quasi-expanding $\delta$-pseudo-orbit
of $f$ in $\mathrm{Cl}_{M^n}(\Lambda)$.

So, from Theorem~\ref{thm2.1} there can be found a periodic point $p$
of $f$ with period
\begin{equation*}
\tau_{p}=m_i+k_i+\cdots+m_j+k_j,
\end{equation*}
which $\varepsilon$-shadows the above $\delta$-pseudo-orbit of $f$
such that $\mathrm{Orb}_{\!f}^+(p)\subset B_{\varepsilon}(\Lambda)$
and
\begin{equation*}
\frac{1}{k}\sum_{i=1}^{k}\log\|D_{\!f^{\tau_{p}-i}(p)}f\|_{\mathrm{co}}\ge\hat{\gamma}-\varepsilon>\gamma^\prime
\end{equation*}
for all $k=1, \ldots, \tau_{p}$.

Since $k_i$ can be chosen arbitrarily large, $\tau_{p}$ can also be
arbitrarily large. The rest is to check that such $p$ satisfies the
abnormal inequality.

In fact, for all $i\le\ell\le j$, by \textbf{2)} and \textbf{5)} above we have
\begin{align*}
\sum_{t=0}^{m_\ell-1}\log\|D_{\!f^t(u_\ell)}f\|_{\mathrm{co}}+\sum_{t=0}^{k_\ell-1}\log\|D_{\!f^t(z_\ell)}f\|_{\mathrm{co}}&\le
m_\ell H+k_\ell\gamma_1^{(\ell+1)}\\
&\le(m_\ell+k_\ell)(\eta+\bar{\gamma}).
\end{align*}
Thus,
\begin{equation*}
\sum_{\ell=i}^j\left\{\sum_{t=0}^{m_\ell-1}\log\|D_{\!f^t(u_\ell)}f\|_{\mathrm{co}}+\sum_{t=0}^{k_\ell-1}\log\|D_{\!f^t(z_\ell)}f\|_{\mathrm{co}}\right\}
\le\tau_{p}\bar{\gamma}.
\end{equation*}
Because $p$ $\varepsilon$-shadows this quasi-expanding pseudo-orbit
string, we obtain that
\begin{equation*}
\sum_{t=0}^{\tau_{p}-1}\log\|D_{\!f^t(p)}f\|_{\mathrm{co}}
\le\tau_{p}(\eta+\bar{\gamma}).
\end{equation*}
Thus
\begin{equation*}
\frac{1}{\tau_{p}}\sum_{t=0}^{\tau_{p}-1}\log\|D_{\!f^t(p)}f\|_{\mathrm{co}}
\le\eta+\bar{\gamma}<\gamma^{\prime\prime}.
\end{equation*}
This ends the proof of Theorem~\ref{thm4.1}.
\end{proof}

\section{Proof of Theorem~\ref{thm1} and local diffeomorphisms of the circle}\label{sec5}%

In this section, we will prove Theorem~\ref{thm1} using the theorems proved before. Then Theorem~\ref{thm2} follows easily from Theorem~\ref{thm1}.

\subsection{Proof of Theorem~\ref{thm1}}\label{sec5.1}%
The first statement of Theorem~\ref{thm1} follows immediately from
Theorem~\ref{thm4.1} with $\Lambda=\mathrm{Per}(f)$ and $c=\lambda$.

The second part of Theorem~\ref{thm1} comes from Theorem~\ref{thm2.3}. In fact, if $\lambda_{\min}(\mu,f)>0$ then from Theorem~\ref{thm2.3}, it follows that $\mathrm{supp}(\mu)\subseteq\mathrm{Cl}_{M^n}(\mathrm{Per}(f))$; and so from the first part of Theorem~\ref{thm1} proved, we can obtain that
\begin{equation*}\begin{split}
\lambda_{\min}(\mu,f)&=\lim_{k\to+\infty}\frac{1}{k}\log\|D_{\!x}f^k\|_{\mathrm{co}}\qquad\mu\textrm{-a.e. }x\in\mathrm{Cl}_{M^n}(\mathrm{Per}(f))\\
&=\lim_{k\to+\infty}\frac{1}{k}\log\min\left\{\|(D_{\!x}f^k)v\|;\;v\in T_{\!x}{M^n}\textrm{ and }\|v\|=1\right\}\\
&\ge\lim_{k\to+\infty}\frac{1}{k}\log(C\exp(k\lambda))\\
&=\lambda.
\end{split}\end{equation*}

And the third part of Theorem~\ref{thm1} trivially holds from the first statement of this theorem. In fact, it follows from statement (1) of Theorem~\ref{thm1} that $f$ is uniformly expanding on $\Omega(f)$. As all ergodic measures of $f$ are supported on $\Omega(f)$, it follows from Ma\~{n}\'{e}'s
criterion \cite[Lemma~I-5]{Man87} as mentioned in Section~\ref{sec1.3} that there exists an integer $m\ge1$ and a constant $\lambda^\prime>0$ such that
\begin{equation*}
\int_{M^n}\log\|D_{\!x}f^m\|_{\mathrm{co}}\,d\mu\ge\lambda^\prime
\end{equation*}
for all ergodic measures $\mu$ of $f$ supported on $M^n$.
Thus, $f$ is uniformly expanding on $M^n$ from Ma\~{n}\'{e}'s
criterion once again.

This completes the proof of Theorem~\ref{thm1}.
\subsection{Proof of Theorem~\ref{thm2}}\label{sec5.2}%
Let
\begin{equation*}
T_{\!x}M^n=E(x)\oplus F(x)\quad \textrm{with }\dim E(x)=1\qquad\forall x\in\mathrm{Cl}_{M^n}(\mathrm{Per}(f))
\end{equation*}
be an $(\eta,1)$-dominated splitting given by the hypotheses of Theorem~\ref{thm2}. Then, there exists an integer $m\ge1$ such that
\begin{equation*}
\frac{\|D_{\!x}f^m|{E(x)}\|}{\;\;\,\|D_{\!x}f^m|{F(x)}\|_{\mathrm{co}}}\le \frac{1}{2}\qquad\forall x\in\mathrm{Cl}_{M^n}(\mathrm{Per}(f)).
\end{equation*}
As the minimal Lyapunov exponent $\lambda_{\min}(x,f)>0$ and is uniformly bounded away from $0$ for $x\in\mathrm{Per}(f)$, from $\dim E(x)=1$ it follows that
\begin{equation*}
\lambda_{\min}(x,f^m)=\lim_{N\to\infty}\frac{1}{N}\sum_{\ell=0}^{N-1}\log\|D_{f^{\ell m}(x)}f^m|E\|=m\lambda_{\min}(x,f)\ge\lambda
\end{equation*}
for some constant $\lambda>0$. Hence $Df^m|E$ and then $Df^m$ are nonuniformly expanding on $\mathrm{Per}(f^m)$. Thus Theorem~\ref{thm1} implies that $f^m$ is uniformly expanding on $\mathrm{Cl}_{M^n}(\mathrm{Per}(f^m))$.

This completes the proof of Theorem~\ref{thm2}.
\subsection{Local diffeomorphisms of the circle}\label{sec5.3}
As another application of our result Theorem~\ref{thm1}, we will consider a local
diffeomorphism of the unit circle in this subsection.

Let $\mathbb{T}^1$ be the unit circle. Using Theorem~\ref{thm1} we can obtain
the following result, which is indeed a special case of Theorem~\ref{thm2}.

\begin{Mthm}\label{thm3}
Let $f\colon \mathbb{T}^1\rightarrow \mathbb{T}^1$ be a $\mathrm{C}^1$-class endomorphism of
$\mathbb{T}^1$ which does not contain any critical points. If every periodic point
of $f$ has a positive Lyapunov exponent and such exponent is
uniformly bounded away from zero, then $f$ is uniformly expanding on
$\mathrm{Cl}_{\mathbb{T}^1}(\mathrm{Per}(f))$ and moreover, for any ergodic measure $\mu$ of $f$, either its support $\mathrm{supp}(\mu)$ is contained in $\mathrm{Cl}_{\mathbb{T}^1}(\mathrm{Per}(f))$ or its Lyapunov exponent is zero.
\end{Mthm}

\begin{proof}
First, we assume $\mathrm{Per}(f)\not=\varnothing$. Then $\mathrm{Per}(f)$ is nonuniformly expanding by $f$ and then
the statement comes immediately from Theorems~\ref{thm1} and \ref{thm2.3}. We notice that from \cite{L73}, it follows that for any ergodic measure $\mu$ of $f$, its Lyapunov exponent $\lambda(\mu)\ge0$.

If $\mathrm{Per}(f)=\varnothing$, then from Theorem~\ref{thm2.3} we see that for any ergodic measure $\mu$ of $f$, its Lyapunov exponent must be zero.

This proves Theorem~\ref{thm3}.
\end{proof}

We here give a remark on the proof of Theorem~\ref{thm3} above. It
is known, from~\cite{Dai09}, that in the $1$-dimensional case Lyapunov
exponent is continuous with respect to ergodic measures in the sense
of weak-$*$ topology. However, although $\mathrm{Per}(f)$ is dense
in $\mathrm{Cl}_{\mathbb{T}^1}(\mathrm{Per}(f))$, one still cannot guarantee, without any generic condition, that
every ergodic measure of $f$ supported on
$\mathrm{Cl}_{\mathbb{T}^1}(\mathrm{Per}(f))$ can be arbitrarily approximated by
periodic measures. So, the proof of Theorem~\ref{thm3} presented
above is of interest itself.

In the situation of Theorem~\ref{thm3}, if $f$ does not have any periodic points, then from Theorem~\ref{thm3} we see that
\begin{equation*}
|f_{f^n(x)}^\prime\cdots f_x^\prime|^{1/n}\to 1\quad \textrm{as }n\to+\infty
\end{equation*}
uniformly for $x\in\mathbb{T}^1$.
\section{Appendix: closing up quasi-expanding strings}\label{sec6}

In the section, we will prove Theorem~\ref{thm2.1} stated in
Section~\ref{sec2.1} following the standard way, see Gan~\cite{G}, for
example.

For this, we need a simple sequence version of shadowing lemma
borrowed from~\cite{G}. In the following lemma, we let
$Y=\left\{(x_i)_{i\in\mathbb{Z}}\,|\,x_i\in X_i\right\}$ where $X_i$
is an $n$-dimensional Euclidean space endowed with the norm
$\|\cdot\|_i$ for every $i$. Under the supremum norm
$\|y\|=\sup_{i\in\mathbb{Z}}\|x_i\|_i$ for $y=(x_i)$, $Y$ is a
Banach space. We only consider the mapping $\Phi\colon Y\rightarrow
Y$ which has the form
\begin{equation*}
(\Phi y)_{i+1}=\Phi_i x_i\qquad\textrm{where }\Phi_i\colon X_i\rightarrow X_{i+1}\quad\forall i\in\mathbb{Z}.
\end{equation*}
For any $r>0$, let $X_i(r)=\{x_i\in X_i;\ \|x_i\|_i\le r\}$.

Now, the sequence version of shadowing lemma for expanding
pseudo-orbit can be described as follows:

\begin{lemma}\label{lem7.1}
Let there be given numbers $\gamma\in(0,1), \epsilon>0$ with
\begin{equation*}
\epsilon_1:=2\epsilon(1+\gamma)/(1-\gamma)<1
\end{equation*}
and $r>0$. Let
$\varsigma=\frac{(1-\gamma)(1-\epsilon_1)}{2(1+\gamma)}$ and
$\delta\in(0,{r}{\varsigma}]$. Assume
$\Phi=(\Phi_i)_{i\in\mathbb{Z}}\colon Y\rightarrow Y$ has the
form
\begin{equation*}
\Phi_i=H_i+\phi_i\colon X_i(r)\rightarrow X_{i+1}
\end{equation*}
where $H_i\colon X_i\rightarrow X_{i+1}$ is a linear isomorphism. If
there holds that
\begin{equation*}
\|H_i\|_{\mathrm{co}}\ge\gamma^{-1},\quad \mathrm{Lip}(\phi_i)\le
\varsigma,\quad
and\quad\|\phi_i(\mathbf{0})\|\le\delta\qquad\forall i\in\mathbb{Z},
\end{equation*}
then $\Phi$ has a unique fixed point $v$ in $Y$ with
$\|v\|\le\delta\varsigma^{-1}$.
\end{lemma}

This statement is a simple consequence of Gan~\cite[Theorem~2.3]{G}.
So we omit its proof here.

Let $\gamma\in(0,1)$ be arbitrarily given. A positive number string
$(b_i)_{i=0}^{\ell-1}$ of length $\ell\ge1$, is called
$\gamma$-expanding if there holds $b_i\ge\gamma^{-1}$ for all
$i=0,\ldots,\ell-1$. It is called $\gamma$-quasi-expanding if there
holds the condition: $\prod_{i=1}^{k} b_{\ell-i}\ge \gamma^{-k}$ for
all $k=1,\ldots,\ell$.

A string of positive numbers $(c_i)_{i=0}^{\ell-1}$ is called
well-adapted to a $\gamma$-quasi-expanding string
$(b_i)_{i=0}^{\ell-1}$, provided that $\prod_{i=0}^{\ell-1}c_i=1$
and $\prod_{i=0}^{k}c_i\le1$ for $k=0,\ldots,\ell-2$ if $\ell\ge 2$
and $(b_i/c_i)_{i=0}^{\ell-1}$ is $\gamma$-expanding.

Then, the following is a special case of the combinatorial lemma of
Liao~\cite{L79}, also see \cite{G, Dai-TMJ}.

\begin{lemma}\label{lem7.2}
Let $\gamma\in(0,1)$ be arbitrarily given. Any
$\gamma$-quasi-expanding string $(b_i)_{i=0}^{\ell-1}$ of length
$\ell\ge 2$ has a well-adapted string $(c_i)_{i=0}^{\ell-1}$ such
that $\min\{\gamma b_i, 1\}\le c_i\le b_i$ for all $0\le i<\ell$.
\end{lemma}

The following lemma is standard:

\begin{lemma}\label{lem7.3}
Given any $f\in\mathrm{Diff}_{\mathrm{loc}}^1(M^n)$. For any
$\varepsilon,\tau,\hat{\varsigma}>0$ there exists a number $r$ with
$0<r\le\varepsilon$ such that if $x,y\in M^n$ satisfy
$\textsl{d}(f(x),y)\le r$, then the lift of $f$ at $(x,y)$
\begin{equation*}
\Phi_{\!x\rightsquigarrow y}=\exp_{\!y}^{-1}\circ
f\circ\exp_{\!x}\colon T_{\!x}M^n(r)\rightarrow T_{\!y}M^n
\end{equation*}
can be well defined such that $\Phi_{\!x\rightsquigarrow
y}=H_{\!xy}+\phi_{\!xy}$ where
$\mathrm{Lip}(\phi_{\!xy})\le\hat{\varsigma}$ and where $H_{\!xy}$
is a linear isomorphism satisfying
\begin{equation*}
1-\tau\le\frac{\|H_{\!xy}\|_{\mathrm{co}}}{\|D_{\!x}f\|_{\mathrm{co}}}\le
1+\tau.
\end{equation*}
\end{lemma}

In what follows, let
\begin{equation*}
K={\sup}_{x\in M}\left\{\|D_{\!x}f\|,
\|(D_{\!x}f)^{-1}\|\right\}.
\end{equation*}

Now, we are ready to prove the theorem.

\begin{proof}[Proof of Theorem~\ref{thm2.1}]
We need to prove only the closing property. Let $\lambda>0$ and
$\varepsilon>0$ be arbitrarily given as in the statement of
Theorem~\ref{thm2.1}.

Let
$\left(\left(x_i,f^{n_i}(x_i)\right)\right)_{-\infty}^{+\infty}$
be a $\lambda$-quasi-expanding $\delta$-pseudo-orbit of $f$  in $M^n$
where $\delta>0$ be arbitrarily given; i.e.,
$\left(x_i,f^{n_i}(x_i)\right)$ is a $\lambda$-quasi-expanding string
of length $n_i$ with $\textsl{d}(f^{n_i}(x_i),x_{i+1})<\delta$ and
$n_i\ge 1$ for each $i\in\mathbb{Z}$. Write
\begin{equation*}
N_i=\left\{
\begin{array}{ll}
0& \textrm{if }i=0,\\
n_0+n_1+\cdots+n_{i-1} & \textrm{if }i>0,\\
-n_i-n_{i+1}-\cdots-n_{-1}& \textrm{if }i<0.
\end{array}
\right.
\end{equation*}
Let $y_j=f^{j-N_i}(x_i)$ and $X_j=T_{\!y_j}M^n$ for any $N_i\le
j<N_{i+1}$ and any $i\in\mathbb{Z}$. Then $(y_j)_{j\in\mathbb{Z}}$
is a $\delta$-pseudo-orbit of $f$ in $M^n$.

Next, we will $\varepsilon$-shadow $(y_j)_{j\in\mathbb{Z}}$ by a
real orbit of $f$ if $\delta$ is sufficiently small.

It is easily seen that there can be found two numbers $\tau\in(0,1)$
and $\gamma\in(0,1)$ such that
\begin{equation*}
(1-\tau)\exp{\lambda}\ge \gamma^{-1}.
\end{equation*}
We now take $\epsilon>0$ small enough to satisfy
\begin{equation*}
\epsilon_1:=\frac{2\epsilon(1+\gamma)}{1-\gamma}<1.
\end{equation*}
Let
\begin{equation*}
\varsigma=\frac{(1-\gamma)(1-\epsilon_1)}{2(1+\gamma)}\quad
\textrm{and}\quad \delta\in(0,r\varsigma]
\end{equation*}
where $r$ is determined by Lemma~\ref{lem7.3} in correspondence with
the triplet $(\varepsilon,\tau,\hat{\varsigma})$ where
$\hat{\varsigma}=\varsigma/(K\exp \lambda)$. Then, according to
Lemma~\ref{lem7.3} the lift $\Phi_{\!y_j\rightsquigarrow y_{j+1}}$
of $f$ at $(y_j,y_{j+1})$
\begin{equation*}
\Phi_j:=\exp_{\!y_{j+1}}^{-1}\circ f\circ\exp_{\!y_j}\colon
X_j(r)\rightarrow X_{j+1}\qquad (N_i\le j\le N_{i+1}-1)
\end{equation*}
has the form $\Phi_j=H_j+\phi_j$ such that
\begin{equation*}
\|H_j\|_{\mathrm{co}}\ge(1-\tau)\|D_{\!y_j}f\|_{\mathrm{co}},\quad
\mathrm{Lip}(\phi_j)\le\frac{\varsigma}{K\exp\lambda}\qquad
\end{equation*}
for $j=N_i,\ldots,N_{i+1}-1$, and
\begin{equation*}
\phi_j(\textbf{0})=\textbf{0}\qquad\textrm{for
}j=N_i,\ldots,N_{i+1}-2,
\end{equation*}
and
\begin{equation*}
\|\phi_j(\textbf{0})\|\le\delta\qquad \textrm{for }j=N_{i+1}-1.
\end{equation*}
So, $\left(\|H_j\|_{\mathrm{co}}\right)_{j=N_i}^{N_{i+1}-1}$ is a
$\gamma$-quasi-expanding string because the string
$\left(\|D_{\!y_j}f\|_{\mathrm{co}}\right)_{j=N_i}^{N_{i+1}-1}$ is
$e^{-\lambda}$-quasi-expanding by the hypothesis of the theorem. And
hence corresponding to it, there can be found from
Lemma~\ref{lem7.2} a well-adapted string
$(c_j)_{j=N_i}^{N_{i+1}-1}$ of length $N_{i+1}-N_i$ such that
$K^{-1}\exp(-\lambda)\le c_j\le K$ for all $N_i\le j\le N_{i+1}-1$.

For any $i\in\mathbb{Z}$ and any $N_i\le j\le N_{i+1}-1$ let
\begin{equation*}
g_j=\prod_{k=N_i}^{j}c_k,\quad\widetilde{H}_j=c_j^{-1}H_j, \quad
\tilde{\phi}_j(v)=g_j^{-1}\phi_j(g_{j-1}v),
\end{equation*}
and further define
\begin{equation*}
\widetilde{\Phi}_j=\widetilde{H}_j+\tilde{\phi}_j\colon
X_j\rightarrow X_{j+1}.
\end{equation*}
Denote by $\Psi_j=\Phi_j\cdots\Phi_{N_i}$ and
$\widetilde{\Psi}_j=\widetilde{\Phi}_j\cdots\widetilde{\Phi}_{N_i}$
for $N_i\le j\le N_{i+1}-1$. Then we have
$\widetilde{\Psi}_j=g_j^{-1}\Psi_j$. Note that
$g_{N_{i+1}-1}^{}=1$ and
$\widetilde{\Psi}_{N_{i+1}-1}=\Psi_{N_{i+1}-1}$ for any $i$.

Thus,
$\mathrm{Lip}(\tilde{\phi}_j)=g_j^{-1}\mathrm{Lip}({\phi}_jg_{j-1})=c_j^{-1}\mathrm{Lip}({\phi}_j)\le
\varsigma$ for all $j$ and $\tilde{\phi}_j(\textbf{0})=\mathbf{0}$
for any $N_i\le j<N_{i+1}-1$ and
$\|\tilde{\phi}_j(\textbf{0})\|=\|\phi_j(\mathbf{0})\|\le\delta$ for
$j=N_{i+1}-1$. Then, according to Lemma~\ref{lem7.1},
$\widetilde{\Phi}=(\widetilde{\Phi}_j)\colon Y(r)\rightarrow Y$
where $Y=\prod_{j\in\mathbb{Z}}X_j$, has a unique fixed point
$\tilde{v}=(\tilde{v}_j)$ such that $\|\tilde{v}\|\le
\delta\varsigma^{-1}<\varepsilon$.

Let $v_{N_i}=\tilde{v}_{N_i}$ for all $i$ and we recursively define
\begin{equation*}
v_j=\Phi_{j-1}(v_{j-1})\qquad \forall N_i<j<N_{i+1}-1.
\end{equation*}
To ensure this, we need to check that
$\|v_j\|\le\delta\varsigma^{-1}$, i.e. $v_{j-1}\in X_{j-1}(r)$.
Indeed, since
\begin{equation*}
v_j=\Psi_{j-1}(v_{N_i})=g_{j-1}\widetilde{\Psi}_{j-1}(v_{N_i})=g_{j-1}\tilde{v}_j,
\end{equation*}
we have $\|v_j\|\le\|\tilde{v}_j\|\le\delta\varsigma^{-1}$. From
\begin{equation*}
v_{N_i+1}=\tilde{v}_{N_i+1}=\widetilde{\Psi}_{N_{i+1}-1}(\tilde{v}_{N_i})=\Psi_{N_{i+1}-1}(v_{N_i})=\Phi_{N_{i+1}-1}(v_{N_{i+1}-1}),
\end{equation*}
we see that $v=(v_j)$ is a fixed point of $\Phi=(\Phi_j)$ and
$\|v\|\le \delta\varsigma^{-1}<\varepsilon$. Let
$z=\exp_{\!y_0}(v_0)$. Thus, $z$ can $\delta\varsigma^{-1}$-shadow
$\{y_j\}$.

Now let
$\left(\left(x_i,f^{n_i}(x_i)\right)\right)_{-\infty}^{+\infty}$
be periodic, i.e., there is some $k\ge 0$ such that $x_{i+k+1}=x_i$
and $n_{i+k+1}=n_{i}$ for all $i$. Define $\tilde{w}$ in the way
$(\tilde{w})_j=(\tilde{v})_{N_{k+1}+j}$ for any $j\in\mathbb{Z}$.
Since both $\tilde{v}$ and $\tilde{w}$ are fixed points of
$\widetilde{\Phi}$ in $Y(\delta\varsigma^{-1})$,
$\tilde{v}=\tilde{w}$ by the uniqueness. Thus, $v=w$ and further $z$
has period $N_{k+1}$.

This ends the proof of Theorem~\ref{thm2.1}.
\end{proof}

\end{document}